\documentclass[12pt, reqno]{amsart}
\pdfoutput=1

\usepackage{fancyhdr}
\usepackage{hyperref}
\usepackage[margin=0.8in]{geometry}

\usepackage{amsmath, amssymb, amsfonts, amsrefs, graphicx}
\usepackage{amsthm}
\usepackage{cleveref} % adds amsthm labels to refs
\usepackage{calligra}
\usepackage{mathrsfs}
\usepackage{tikz-cd}
\usepackage{bbm}

\renewcommand{\eprint}[1]{\href{https://arxiv.org/abs/#1}{{\tt arXiv:#1}}}

\BibSpec{article}{%
 +{} {\PrintAuthors} {author}
 +{,} { \textit} {title}
 +{.} { } {part}
 +{:} { \textit} {subtitle}
 +{,} { \PrintContributions} {contribution}
 +{.} { \PrintPartials} {partial}
 +{,} { } {journal}
 +{} { \textbf} {volume}
 +{} { \PrintDatePV} {date}
 +{,} { \issuetext} {number}
 +{,} { \eprintpages} {pages}
 +{,} { } {status}
 +{,} { \PrintDOI} {doi}
 +{,} { \eprint} {eprint}
 +{} { \parenthesize} {language}
 +{} { \PrintTranslation} {translation}
 +{;} { \PrintReprint} {reprint}
 +{.} { } {note}
 +{.} {} {transition}
 +{} {\SentenceSpace \PrintReviews} {review}
 }

\newcommand{\R}{\mathbb{R}}

\newcommand{\C}{\mathbb{C}}
\newcommand{\Z}{\mathbb{Z}}

\renewcommand{\Im}{\text{Im}}

\newcommand{\ms}{\mathscr}
\newcommand{\mc}{\mathcal}
\newcommand{\mf}{\mathfrak}
\newcommand{\oh}{\mathcal{O}}

\newcommand{\bi}{\bar{i}}

\newcommand{\pbar}{\bar{\partial}}

\newcommand{\genK}{\mc{J}_\omega}

\DeclareMathOperator{\GL}{GL}

\DeclareMathOperator{\Spin}{Spin}

\DeclareMathOperator{\im}{Im}

\DeclareMathOperator{\Pic}{Pic}

\DeclareMathOperator{\hol}{hol}
\DeclareMathOperator{\Ann}{Ann}

\DeclareMathOperator{\curv}{curv}

\DeclareMathOperator{\Pf}{Pf}
\DeclareMathOperator{\Pfr}{Pfr}

\DeclareMathOperator{\End}{End}
\DeclareMathOperator{\Diff}{Diff}

\DeclareMathOperator{\Hom}{Hom}
\DeclareMathOperator{\Alt}{Alt}

\DeclareMathOperator{\shfHom}{\mathscr{H}\textnormal{\kern -3pt {\calligra\large om}}\,}
\DeclareMathOperator{\shfEnd}{\mathscr{E}\textnormal{\kern -3pt {\calligra\large nd}}\,}
\DeclareMathOperator{\Ext}{Ext}

\DeclareMathOperator{\shfExt}{\mathscr{E}\textnormal{\kern -3pt {\calligra\large xt}}\,}

\DeclareMathOperator{\IF}{IF}
\DeclareMathOperator{\HF}{HF}

\newcommand{\blank}{{-}}

\newcommand{\la}{\langle}
\newcommand{\ra}{\rangle}

\newtheorem*{theorem*}{Theorem}
\newtheorem{theorem}{Theorem}[section]
                       
\newtheorem{corollary}[theorem]{Corollary}               
\newtheorem{proposition}[theorem]{Proposition}

\theoremstyle{definition}
\newtheorem{definition}[theorem]{Definition}             
\newtheorem{example}[theorem]{Example}                   
\newtheorem{remark}[theorem]{Remark}                   

\title{Supersymmetric Topological Sigma Models and Doubling Spaces}
\author{Daniel M.\ Halmrast}

\begin{document}

%--- Abstract
\begin{abstract}

Witten's topological B-model on a Calabi-Yau background is known to reproduce,
in the open string sector, the derived category of coherent sheaves.
When the target space is a complex torus, the topological model enjoys
a non-geometric symmetry known as T-duality, which relates the theories
on the torus and dual torus backgrounds.
By considering the ``double field theory'' of Hull and Reid-Edwards on the
product of a torus and its dual, T-duality occurs as a geometric symmetry
of the target space.

Building on the methods of Y.\ Qin, we propose a method of analyzing the
topological B-model on a torus in the doubled geometry framework which
naturally incorporates certain rank-one D-branes, providing a different
perspective on the derived category of the torus. In certain cases, the
intersection theory of these lifted branes correctly computes the BRST
cohomology of the B-model, and hence the derived Hom-spaces of the
corresponding line bundles. 
\end{abstract}

\maketitle

\tableofcontents

%=== Introduction 
\section{Introduction}

During the second superstring revolution of the 90's, many intricate and
surprising relationships were found between the physics of string theory and
the mathematics of complex geometry. Witten's topological string theory
models \cite{witten1991mirror} furnished a deep connection between the
low-energy physics of certain carefully-engineered string theory and certain
invariants of the complex manifold on which it propagates.

The result was extended to include exotic objects called ``D-branes'',
which appear as boundary states in the open string theory. Although these
naturally appear as coherent sheaves on the target manifold, 
Sharpe \cite{Sharpe:1999qz} found that these D-branes may appear
as more general objects in the bounded derived category of
coherent sheaves. The equivalence of categories between the category of
topological D-branes and the bounded derived category of
coherent sheaves was constructed by Aspinwall and Lawrence
\cite{aspinwall2001derived}, and the predicted connection between the
low-energy states of the physical theory and the ext-groups appearing in the
derived category was later computed by Katz and Sharpe
\cite{katz2002d}.

The dictionary between the
topological nonlinear sigma model of \cite{witten1991mirror} (and its
categorization in the sense of \cite{Sharpe:1999qz, aspinwall2001derived}) and the complex geometry
of the target space is not one-to-one. Rather, the string theory carries additional
structure that the target space geometry does not, and interpreting this additional
structure in terms of the geometry leads to many deep insights into the derived
category, including the far-reaching homological mirror symmetry conjecture of
Kontsevich \cite{kontsevich1995homological} and the intricate structures of
Bridgeland stability conditions \cite{bridgeland2007stability},
Donaldson-Thomas invariants \cite{Kontsevich:2008fj}, and Joyce structures
\cite{Joyce:2008pc}.

One such structure on the string theory appears in nonlinear sigma models into
geometries which are torus fibrations. The principle of ``T-duality'' states
that the string theory on a torus fibration is identical, via field
redefinitions, to the string theory on the dual fibration. Using T-duality, it
is possible to construct a string theory which locally behaves like a nonlinear
sigma model into a torus fibration, but which admits no global geometric
description \cite{Hull:2004in}. The doubled formalism, appearing concretely in
\cite{Siegel:1993th}\cite{Siegel:1993xq}, and formalized in \cite{Hull:2009mi},
attempts to geometrize T-duality by considering an associated string theory on
the product of a torus and its dual. 

Recently, the topological A-model was examined using this framework
\cite{qin2020coisotropic}.
There, Y.\ Qin proposed a way to double not only the ambient geometry
but also the A-type D-branes on the torus. He found that the intersection
theory of the doubled D-branes is closely related to the Fukaya category
of the original torus.

In this paper, we propose an extension of the ideas of \cite{qin2020coisotropic}
to the topological B-model, and explore the interplay between doubled geometry
and generalized complex geometry.
Section \ref{chapter:physics} covers the physics background which motivates
the main construction. We review the basics of the worldsheet perspective on
string theory, and the interplay between the worldsheet perspective and
target space geometry. T-duality is introduced, and an example of
a string theory with a non-geometric background is examined in detail.
We also see how supersymmetry on the worldsheet constrains the geometry
of the target space.

Section \ref{chapter:math} covers the prerequisite mathematics for the main
construction. Of primary importance is the theory of connections on line bundles,
which we approach from the analytic viewpoint. The key tool is the theory of
Cheeger-Simons differential characters, which is a refinement of cohomology
that tracks holonomy of $U(1)$ connections.
The theory simplifies beautifully on the torus, and we analyze the
analytic description of line bundles on tori, and their cohomology.
Here we introduce the notion of $E$-flatness, a condition on the connection
of a complex line bundle. The space of $E$-flat connections naturally
forms a torsor over the dual torus, and has a particularly nice description
in terms of differential characters.
Finally, we examine the theory of generalized complex geometry in the
sense of \cite{gualtieri2004generalized}, and the connection between generalized
submanifolds and $U(1)$-bundles.

\subsubsection*{Main Results}
The main results lie in section \ref{chapter:main}. We construct the doubling
torus following the ideas of \cite{qin2020coisotropic}, and propose a new
definition of a lift $\mathbb{L}$ of $E$-flat rank-one space-filling D-branes in the
topological B-model. This lift is naturally expressed in terms of differential
characters, and we show that our definition of a lift agrees with
\cite{qin2020coisotropic}. These lifts are naturally holomorphic Lagrangian
subtori under the canonical structures on the doubling space, and we explicitly
compute their intersection properties.

A concrete link is then established between the generalized complex geometry of
the base torus $X$ and the (ordinary) complex geometry of the doubled torus 
$\mathbb{X}$.
\begin{theorem*}[Theorem~\ref{thm:doubledTangentIsGCTangent}]
        The tangent bundle of $\mathbb{X}$ is canonically isomorphic
        (under $\sigma$) to the pullback of the sum of tangent and cotangent bundles
        of either $X$ or $\hat{X}$. That is,
        \begin{equation}
            T\mathbb{X}\cong \pi^*(TX\oplus T^*X)\cong \hat{\pi}^*(T\hat{X}\oplus T^*\hat{X})
        \end{equation}
        as smooth vector bundles.
\end{theorem*}
We also establish that our definition of a lift $\mathbb{L}(\mc{L})$ 
of a line bundle $\mc{L}$ lifts the generalized tangent bundle defined by
the line bundle to the tangent bundle of the Lagrangian lift.
\begin{theorem*}[Theorem~\ref{prop:tangentBundleLift}] 
        The tangent bundle of $\mathbb{L}(\mc{L})$ is the subbundle of $T\mathbb{X}$
        defined by
        \begin{equation}
            T\mathbb{L}(\mc{L}) = \{ (v,\hat{v})\in TX\oplus T\hat{X}\ |\ 
            \iota_{\hat{v}}\sigma = \iota_v E\}
        \end{equation}.
        That is, under the isomorphism $T\mathbb{X}\cong \pi^*(TX\oplus T^*X)$ the
        tangent bundle of $\mathbb{L}(\mc{L})$ is isomorphic to the pullback of
        the generalized tangent bundle of $(X,E)$.
\end{theorem*}

The main connection to physics is found in \cref{SSS:relationPhysicsExt},
where we make explicit connections to
the derived category of the topological B-model. 
In particular, we establish agreement between a certain holomorphic refinement
of Floer intersection
theory on the doubling space and the derived $\Hom$ on the base.
\begin{theorem*}[Theorem~\ref{thm:ExtIsIntersectionEqualChern}]
        Let $\mc{L}_1$ and $\mc{L}_2$ be two holomorphic line bundles with
        $c_1(\mc{L}_1) = c_1(\mc{L}_2)$. Then,
        \begin{equation}
            \Hom_B(\mc{L}_1,\mc{L}_2) 
            = \HF^*_\mc{J}(\mathbb{L}(\mc{L}_1),\mathbb{L}(\mc{L}_2))
        \end{equation}
        where $\HF^*_\mc{J}$ is the $\mc{J}$-holomorphic part of $\HF^*$,
        $\Hom_B$ is the B-model open string spectrum,
        and $\mc{J} = \mc{J}_J$ is the complex structure on $\mathbb{X}$
        induced by the complex structure on $X$ (\cref{ex:liftComplexStructure}).
    \end{theorem*} 

\subsubsection*{Acknowledgements}
This work is largely based off of my doctoral thesis
\cite{halmrast2024supersymmetric} which was completed at the University of
California, Santa Barbara. I would like to thank in particular my advisor,
Dave Morrison, for teaching me all about the relevant string theory and complex
geometry for this problem, and for the invaluable insight and support throughout.

%=== Chapter Physics 
\section{Physics Background}
\label{chapter:physics}
Despite the main construction of this paper being purely mathematical
in nature, the motivation of the construction stems mainly from notions
in mathematical physics, namely string theory.
We will, therefore, take a detour through the landscape of string theory and
the primary motivating constructions found therein. We focus on two
independent constructions: \textit{T-duality}, a symmetry of the string theory
that is distinctly non-geometrical, and \textit{supersymmetry},
which provides the link between the physics and derived categories.

We follow the excellent exposition of \cite{Johnson:2023onr} for this section.

\subsection{The Bosonic String}
\label{SSE:bosonicString}

In the worldsheet perspective on string theory, the fundamental space-time
is taken to be the string worldsheet, a two-dimensional (one spatial and one
temporal) smooth manifold $\Sigma$ whose boundary components are all
diffeomorphic to $S^1$. It is imagined that the boundary components are ``incoming''
and ``outgoing'' closed string states, and the field theory computes the physics
of the interaction described by the worldsheet topology.

%--- Free Bosonic String 
\subsubsection{The Free Bosonic String}

The first physical theory to consider is the theory of a free scalar boson on
the worldsheet, which we review now. Take $\Sigma = S^1\times \R$
coordinatized by $(\sigma,\tau)$, with $\sigma$ periodic of period $2\pi$. 
Also, fix a pseudo-Riemannian metric $\gamma_{ab}$ on $\Sigma$ for which
$\sigma$ is spatial and $\tau$ is temporal. 
Then, let $X:\Sigma\to \R$ be a scalar function on $\Sigma$, which we take to be
a dynamical object. These objects are governed by the Polyakov action, a local
functional of $\gamma$, $X$ and their derivatives, defined as\footnote{
    We take throughout this section the convention of Einstein summation, where repeated indices
are to be summed over.}
    \begin{equation}
        S[\gamma,X] = \frac{-1}{4\pi \alpha'}\int_\Sigma 
        \left( \sqrt{-\gamma}\gamma^{ab}\partial_a X\partial_bX \right)dA
    \end{equation}
where $\alpha'$ is a positive real number known as the ``coupling constant''. 
The action governs the entire quantum field theory, and if the action is left
unchanged by some group action on $X$ or $\gamma$, that group action is 
said to be a \textit{symmetry} of the system. 

In fact, this action enjoys many symmetries. 
\begin{itemize}
    \item The group $\Diff_0(\Sigma)$ of small diffeomorphisms
    of $\Sigma$ acts locally on $X$ and $\gamma$. Given a vector field $\xi^a$,
        \begin{equation}
            \begin{aligned}
                X &\mapsto X+\xi^a\partial_a X\\
                \gamma^{ab} &\mapsto \gamma^{ab} + \xi^c\partial_c\gamma^{ab} 
                - \partial_c\xi^a\gamma^{cb}-\partial_c\xi^b\gamma^{ac}\\
            \end{aligned}
        \end{equation}
    \item The group $C^{\infty}(\Sigma)$ acts on the metric by Weyl scaling.
        Given a function $\omega:\Sigma\to \R$,
            \begin{equation}
                \begin{aligned}
                    X &\mapsto X\\
                    \gamma_{ab} &\mapsto e^{2\omega}\gamma_{ab}
                \end{aligned}
            \end{equation}
    \item The translation group $\R$ of $\R$ acts
        on the field $X$ by the natural action. Namely, for $a\in \R$,
            \begin{equation}
                \begin{aligned}
                    X&\mapsto X+a\\
                    \gamma^{ab}&\mapsto\gamma^{ab}
                \end{aligned}
            \end{equation}
\end{itemize}

To find the equations of motion of $X$, and hence the classical physics of this
system, the variational equation $\delta S = 0$ is imposed. Direct computation
reveals that the fields $X$ and $\gamma$ satisfy
\begin{equation}
    \sqrt{-\gamma}\nabla^2 X = 0
\end{equation}.
Using the action of Weyl scaling and $\Diff_0(\Sigma)$, the metric $\gamma$
can always be brought into standard form
\begin{equation}
    \gamma = 
    \begin{bmatrix}
        -1 &0\\
        0  &1
    \end{bmatrix}
    e^{\phi}
\end{equation}
where $e^\phi$ is a global scaling function known as the conformal parameter. 
In this choice of metric, the equations of motion become
\begin{equation}
    \left( \partial_\sigma^2 - \partial_\tau^2 \right)X = 0
\end{equation}.
This is easily solved using characteristics, from which we see that $X$
decomposes into left-moving and right-moving modes
\begin{equation}
    X(\sigma,\tau) = X_L(\tau+\sigma) + X_R(\tau-\sigma)
\end{equation}
. Imposing periodicity $X(\sigma+2\pi,\tau) = X(\sigma,\tau)$ allows the
solutions to be expanded in a Fourier series. Setting $\sigma^+ = \tau+\sigma$ and
$\sigma^- = \tau-\sigma$,
\begin{equation}
    \label{eq:classicalBosonicSolnLR}
    \begin{aligned}
        X_R(\sigma^-) &= \frac{1}{2}X_0 + \alpha' P_R(\sigma^-)
        +i\left( \frac{\alpha'}{2} \right)^{\frac{1}{2}}\sum_{n\neq 0} a_n e^{-in(\sigma^-)}\\
        X_L(\sigma^+) &= \frac{1}{2}X_0 + \alpha' P_L(\sigma^+)
        +i\left( \frac{\alpha'}{2} \right)^{\frac{1}{2}}\sum_{n\neq 0}\frac{1}{n} \tilde{a}_n
        e^{in(\sigma^+)}\\
    \end{aligned}
\end{equation}
where $X_0,P_{L,R}$ are constants, and $a_n$ and
$\tilde{a}_n$ are the Fourier coefficients. To enforce the reality conditions
${X}^*_R = X_R$ and ${X}^*_L = X_L$, the Fourier coefficients must 
satisfy
    \begin{equation}
        \begin{aligned}
            a_{-n} &= (a_n)^*\\ 
            \tilde{a}_{-n} &= (\tilde{a}_n)^*
        \end{aligned}
    \end{equation}.
The complete solution is then
\begin{equation}
	\label{eq:classicalBosonicSolnTotal}
	X = X_L+X_R = X_0 + \alpha'(P_L+P_R)\tau + \alpha'(P_L-P_R)\sigma + \text{oscillators}
\end{equation}
and single-valuedness around $\sigma\mapsto\sigma + 2\pi$ enforces $P_L-P_R = 0$.

\subsubsection{Many Free Bosons and Maps to Target Space} %---

Now that the theory of a single free boson has been analyzed, the generalization
to many non-interacting bosons is straightforward. As before, $\Sigma = S^1\times\R$
is coordinatized by $(\sigma,\tau)$, and $\gamma_{ab}$ is a pseudo-Riemannian
metric on $\Sigma$. Now, fix a positive integer $d$ and let $\mu=0,\ldots,d-1$
be an index running from $0$ to $d-1$, and take $d$ scalar maps
\begin{equation}
    X^\mu:\Sigma\to \R
\end{equation}
as the bosonic fields on $\Sigma$. By adding their actions together, 
we obtain a new theory in which the fields have no interaction with each other:
\begin{equation}
    S\left[\gamma,X^\mu\right] 
    = \frac{-1}{4\pi\alpha'}\int_{\Sigma} \left( 
        \sum_{\mu}\sqrt{-\gamma}\gamma^{ab}\partial_a X^\mu\partial_b X^\mu
    \right)dA
\end{equation}.
The action enjoys all the same symmetries as the single free boson, with now
$d$ copies of the Euclidean group of $\R$:
\begin{itemize}
    \item The small diffeomorphism group $\Diff_0(\Sigma)$ acts locally as before.
    \item The Weyl group $C^{\infty}(\Sigma)$ acts on the metric as before.
    \item The product of the translation groups of each bosonic field
        $\R^d$ acts by translation. 
        Given $a^\mu\in\R^d$,
        \begin{equation}
            X^\mu\mapsto X^\mu + a^\mu
        \end{equation}
        still leaves the action invariant.
\end{itemize}
Although this action is the most obvious one to incorporate many bosonic fields,
the fact that it is not invariant under the full Euclidean group
of $\R^d$, but rather only the translation subgroup,
suggests an extension. 

Let $g_{\mu\nu} = \delta_{\mu\nu}$ be the identity matrix. We can insert this
into the action without modifying it
\begin{equation}\label{stringActionWithMetric}
    S\left[\gamma,X^\mu\right] 
    = \frac{-1}{4\pi\alpha'}\int_{\Sigma} \left( 
        \sum_{\mu}\sqrt{-\gamma}\gamma^{ab}\partial_a X^\mu\partial_b X^\nu g_{\mu\nu}
    \right)dA
\end{equation}.
This new action is exactly equivalent to the old action. Now, however,
we allow $g$ to transform under $E(d)$ in exactly the opposite way $X^\mu$ does,
resulting in a $E(d)$-invariant theory.
Specifically, for $\Lambda^\mu_\nu$ an $SO(d)$-transformation,
\begin{equation}
    \begin{aligned}
        X^\mu&\mapsto \Lambda^\mu_\nu X^\nu\\
        g_{\mu\nu} &\mapsto(\Lambda^{-1})_\mu^\rho(\Lambda^{-1})_\nu^\sigma g_{\rho\sigma}
    \end{aligned} 
\end{equation}.
This new $SO(d)$-action leaves the string action \eqref{stringActionWithMetric}
unchanged, so the symmetry group of this theory is enhanced to include all
of $E(d)$ instead of the subgroup of translations.

We emphasize that this model was built from worldsheet considerations,
and is only a theory of free scalar fields on a two-dimensional surface. From the
perspective of the worldsheet, the physics is that of $d$ independent, non-interacting,
chargeless distinguishable massless fields.
With our careful insertion of the identity matrix $g$ and extension of the
action of $SO(d)$, the theory now has a more ``geometric'' interpretation.

Let us now make the leap in identifying $X^\mu$ with coordinates of an 
immersion of $\Sigma$ into $\R^d$, thought of as a Riemannian manifold with
the standard flat metric. Fixing the standard basis for $\R^d$, define
\begin{equation}\label{eq:worldsheetImmersionEquation}
    \begin{aligned}
        \varphi &: \Sigma\to \R^d\\
        \varphi(\sigma,\tau) &= \left( X^1(\sigma,\tau),\ldots,X^d(\sigma,\tau) \right)
    \end{aligned} 
\end{equation}.
With this definition, the action \eqref{stringActionWithMetric} becomes
\begin{equation}\label{stringActionCoordinateFree}
    \begin{aligned}
        S[\gamma,\varphi] &= \frac{-1}{4\pi \alpha'}\int_\Sigma
        \left(\sqrt{-\gamma}\gamma^{ab}\partial_a\varphi^\mu \partial_b\varphi^\nu g_{\mu\nu}\right)\\
        &=\frac{-1}{4\pi \alpha'}\int_\Sigma
        \left(\sqrt{-\gamma}\gamma^{ab}(\varphi^*g)_{ab}\right)\\
    \end{aligned} 
\end{equation}
where $\varphi^*g$ is the pullback of the metric to the worldsheet.
This action is no longer dependant on the choice of orthonormal frame $X^\mu$,
but rather only on the map $\varphi$. Explicitly, due to the $E(d)$-invariance of the
action, switching coordinate systems to a new orthonormal coordinate system
\begin{equation}
    X^\mu\mapsto  \Lambda^\mu_\nu X^\nu + a^\mu
\end{equation}
for any $\Lambda\in SO(d)$ does not change the action. 
Equivalently, we were able to write the action in a coordinate-free manner
in \eqref{stringActionCoordinateFree} precisely due to this coordinate invariance.

With the action in the form of \eqref{stringActionCoordinateFree}, it is apparent
that the metric need not be the homogeneous flat metric on $\R^d$. Any choice
of metric $g$ on $\R^d$ (thought of as a Riemannian manifold) results in a 
well-defined action. 

\subsubsection{Descent to a Manifold} %---
\label{SSS:descentToManifold}

The previous construction reinterprets the theory of $d$ free bosons
as a theory of maps from $\Sigma$ into $\R^d$. The construction did not depend
on the global structure of $\R^d$, and was also local on $\Sigma$ as well.
Instead of treating the model as a global map into $\R^d$, let us instead
take everything to be done locally, and examine what happens.

Let $(X,g)$ be a Riemannian manifold, and 
\begin{equation}
    \begin{tikzcd}
        \coprod_\alpha U_\alpha\arrow[r, "\iota_\alpha"] & X
    \end{tikzcd} 
\end{equation}
an open covering of $X$ by coordinate charts. Define $U_{\alpha\beta}$ 
to be the intersection of $U_\alpha$ with $U_\beta$ in $U_\alpha$.
For any map $\varphi:\Sigma\to X$, define
\begin{equation}
    \varphi_{U_\alpha}:\varphi^{-1}(U_\alpha)\to U_\alpha
\end{equation}
to be the restriction of $\varphi$ to a coordinate chart.

On each $\varphi^{-1}(U_\alpha)$, $\varphi_{U_\alpha}$ maps into
an open subset of $\R^d$, and we can attempt to impose the familiar action of 
$\eqref{stringActionCoordinateFree}$. To keep things local, we instead
define the \textit{Lagrangian density} on this open set as follows.
Choose orthonormal coordinate functions $x^\mu$ on $U_\alpha$ and define
$\varphi^\mu := x^\mu\circ\varphi_{U_\alpha}^\mu$ to be the local $\mu$th coordinate
of $\varphi$. Then, define the local Lagrangian density to be the local worldsheet $2$-form
\begin{equation}
    \mc{L}_\alpha[\varphi_{U_\alpha}] = \sqrt{-\gamma}\gamma^{ab}\partial_a\varphi^{\mu}\partial_b\varphi^{\nu}g_{\mu\nu}dA
\end{equation}.
By nature of the $E(d)$-invariance of this functional, 
this two-form is independent of choice of coordinates. In fact, by $\Diff_0(\Sigma)$-invariance
this is also independent of choice of coordinates on $\Sigma$.

In particular, this means that on the intersections $U_{\alpha\beta}$ on
the target and $\varphi^{-1}(U_{\alpha\beta})$ on $\Sigma$ the worldsheet two-forms
$\mc{L}_{\alpha}$ and $\mc{L}_\beta$ agree.
Using this, we can glue the local densities into a global $2$-form $\mc{L}[\varphi]$
on all of $\Sigma$. Then, we can define the \textit{bosonic nonlinear sigma model}
action to be
\begin{equation}\label{eq:bosonicNLSMAction}
    S[\gamma,\varphi] = \frac{-1}{4\pi \alpha'}\int_\Sigma
    \sqrt{-\gamma}\gamma^{ab}(\varphi^*g)_{ab}dA
\end{equation}.

Let us reflect for a moment on this construction. The starting point was
a free field theory on the worldsheet $\Sigma$, where we thought of the $\Sigma$ as
space-time and the field as a particle on $\Sigma$. However, by careful
manipulation of the action, we found that a theory of $d$ distinguishable bosonic
particles on $\Sigma$ admits the Euclidean group $E(d)$ of $\R^d$ as a symmetry.
This symmetry allows us to reinterpret the worldsheet theory as instead a theory
governing immersions of $\Sigma$ into $\R^d$ via the action \eqref{stringActionCoordinateFree}.
Finally, using the symmetry of the Euclidean group and, in particular, the target
space coordinate-invariance of the action, the theory could be extended to
describe immersions of $\Sigma$ into arbitrary Riemannian manifolds.
From the worldsheet perspective, these are simply theories of $d$ distinguishable
bosonic particles propagating on $\Sigma$. However, from the target-space perspective
of the Riemannian manifold $X$, this is the theory of a propagating string
immersed in $X$.

This can be summarized in a general guiding principle: \textit{symmetries of
the worldsheet theory correspond to gluing data on the target space}. 

\begin{remark}
    The action \eqref{eq:bosonicNLSMAction} is not quite the most general
    one that could be taken. In particular, there is no demand that $g$ be
    a symmetric tensor, only that it is of type $(0,2)$. We could,
    in principle, replace $g$ with a sum $g+B$ where $B\in \mc{A}^2(X)$
    is a two-form on $X$. This term is called the \textit{B-field},
    and setting a nonzero value for the B-field has many interesting ramifications.
    
    For this paper, we will focus on the case $B=0$. However,
    even in this case the ability to shift the B-field to a nonzero value
    will play a role.
\end{remark}

\subsection{T-Duality of Strings on Tori}
    \label{SSE:tDuality}

    We have observed that the worldsheet theory of the string enjoys many
    symmetries. Some of these symmetries correspond to symmetries of the underlying
    worldsheet, like the $\Diff_0(\Sigma)$ and Weyl group symmetries as well
    as the worldsheet supersymmetries, while
    others were symmetries corresponding to the fields themselves, such as the
    $E(d)$-symmetry of $d$ scalar fields.
    We now explore a new symmetry enjoyed by both the bosonic string and the 
    superstring known as T-duality, which is of a different flavor.

    \subsubsection{Bosonic Strings on the Circle} %---
    
    Consider the bosonic sigma model defined in \cref{SSS:descentToManifold},
    and set the target to be the circle $S^1_R$ of radius $R$.
    This is the theory of a single bosonic field $X$ which is $2\pi R$-periodic
    in the sense that
    \begin{equation}
        \label{eq:bosonicIsPeriodicOnTarget}
        X(\sigma,\tau) = X(\sigma,\tau) + 2\pi R
    \end{equation}
    for all $(\sigma,\tau)\in\Sigma$. 

    Let us revisit the classical solution of \eqref{eq:classicalBosonicSolnLR}.
    There, we argued that $P_L=P_R$ since a shift of $\sigma\mapsto\sigma+2\pi$
    resulted in a shift of the field to
    \begin{equation}
        X(\sigma,\tau)\mapsto X(\sigma,\tau) + 2\pi\alpha'\left( P_L-P_R \right)
    \end{equation}.
    However, due to the condition \eqref{eq:bosonicIsPeriodicOnTarget},
    $P_L-P_R$ is no longer zero, but is instead a multiple of $\frac{R}{\alpha'}$.
    A solution with nonzero $P_L-P_R$ is called a \textit{winding mode}, and the
    integer $\frac{\alpha'(P_L-P_R)}{R}$ is the \textit{winding number}.
    The target space momentum $\alpha'(P_L+P_R)$ is also constrained since the
    target manifold is compact: $P_L+P_R$ must be an integral multiple of
    $1/R$.

    A computation of the mass spectrum \cite{Johnson:2023onr}*{Equation~4.15}
    reveals that a state with $P_L+P_R = \frac{n}{R}$ and 
    $P_L-P_R=\frac{2\pi w R}{\alpha'}$ has total mass
    \begin{equation}
        M^2 = \frac{n^2}{R^2} + \frac{w^2R^2}{\alpha'^2} +\text{oscillator energy}
    \end{equation}.
    As the limit $R\to\infty$ is taken, the momentum values become less
    separated, and approach a continuum as one would expect of a flat space solution.
    States of nonzero winding number become more and more massive, requiring more
    and more energy to create. In the limit these states freeze out,
    and the original solutions to the free boson are recovered.

    The limit $R\to 0$ is more interesting. In this limit, the momentum states
    are frozen out as the lowest nonzero momentum states gains more and more
    mass. It is now the winding states that approach a continuum. Effectively,
    a new dimension has appeared in which the ``winding states'' of the original
    circle theory become momentum states on the new dimension. Notice this is
    a purely ``stringy'' phenomenon; particles on $S^1$ can't wrap the circle,
    so although their momentum behaves the same as that of the string there
    is no winding state that appears. Only in the string theory do we see
    the effective new dimension appear.
    
    There is a symmetry here that relates these two modes. To see it,
    rewrite the solutions of \eqref{eq:classicalBosonicSolnLR} as
    \begin{equation}
        \label{eq:classicalTDualString}
        \begin{aligned}
            X_R(\sigma^-) &= \frac{1}{2}\left( X_0 - \hat{X}_0 \right) 
            + \alpha' P_R(\sigma^-)
            +\text{oscillators}\\
            X_L(\sigma^+) &= \frac{1}{2}\left( X_0+\hat{X}_0 \right)
            + \alpha' P_L(\sigma^+)
            +\text{oscillators}
        \end{aligned}
    \end{equation}
    where we have introduced the redundant coordinate $\hat{X}_0$. Now,
    consider the transformation
    \begin{equation}
        \label{eq:TDualTransformation}
        \begin{aligned}
            X_L&\mapsto X_L\\
            X_R&\mapsto -X_R
        \end{aligned}
    \end{equation}
    which sends our solution $X(\sigma,\tau)$ to
    \begin{equation}
        X\mapsto X_L-X_R = \frac{1}{2}\hat{X}_0 + \alpha'(P_L-P_R)\tau
        +`a'(P_L+P_R)\sigma
        +\text{oscillators}
    \end{equation}.
This solution is identical to the original solution
\eqref{eq:classicalBosonicSolnTotal} except
    $X_0$ has been replaced by $\hat{X}_0$, and $P_L+P_R$ has been replaced
    by $P_L-P_R$. Notice now that the momentum is quantized by values of
    $R/\alpha'$ instead of $1/R$. It is a remarkable fact that all aspects
    of the theory, including the quantum theory, also remain unchanged.
    \begin{theorem}
        [\protect{T-Duality on a Circle, \cite{Johnson:2023onr}*{Section~4.2}
        or \cite{Hori:2003ic}*{Section~11.2.2}}]
        \label{thm:TdualityCircle}
        The (quantum) nonlinear sigma model on a circle of radius $R$ is equivalent
        to the nonlinear sigma model on a circle of radius $R' = \alpha'/R$
        under the transformation \eqref{eq:TDualTransformation}.
    \end{theorem}
    
    This symmetry is called \textit{T-duality}, and is a worldsheet symmetry
    of the action, formed in terms of target space quantities, with no
    interpretation as a symmetry of the target space.
    
    %---

    \subsubsection{T-duality for Tori} %---
    \label{SSS:TDualityTori}
    The symmetry of T-duality can be applied to any factor of target space that
    is a circle. That is, if the target space is of the form 
    $M = X\times (S^1_{R_n})^n$, T-duality can be applied to each of the $n$
    circle factors resulting in a theory on circles of radii $R_n' = \alpha'/R_n$.

    This case was originally worked out in \cite{Narain:1985jj}.
    Consider the theory on a torus $T^n = \R^n/2\pi\Lambda$
    (following the conventions of \cite{Johnson:2023onr}*{Section~4.5}).
    The winding modes are parameterized by $\Lambda$, whereas the momentum modes
    are quantized to live on the dual lattice $\Lambda^*\subseteq(\R^n)^*$.
    Choosing a basis $e^a$ for $\Lambda$ and set $e^{*a}$ as the dual basis
    for $\Lambda^*$, let $X^a$ be the corresponding coordinate function on
    the torus, and $P^a$ the dual coordinates on $(\R^n)^*/\Lambda^*$. 
    Then, there are classical solutions for any $\vec{n}=n^a\in\Lambda^*$ and
    $\vec{w}=w^a\in\Lambda$, similar to \eqref{eq:classicalTDualString}:
    \begin{equation}
        \begin{aligned}
            X_R^a(\sigma^-) &= \frac{1}{2}\left(X_0^a - \hat{X}_0^a\right)
            +\alpha' P_R^a(\sigma^-) +\text{oscillations}\\
            X_L^a(\sigma^+) &= \frac{1}{2}\left(X_0^a + \hat{X}_0^a\right)
            +\alpha' P_L^a(\sigma^+) +\text{oscillations}\\
        \end{aligned} 
    \end{equation}
    with
    \begin{equation}
        \begin{aligned}
            P_L^a+P_R^a = n^a\\
            P_L^a-P_R^a = w^a
        \end{aligned} 
    \end{equation}.
    T-duality acts by switching factors between the two lattices, which
    means it naturally acts on the enlarged space $\R^n\times(\R^n)^*$,
    which we equip with the neutral metric
    \begin{equation}
        G = 
        \begin{bmatrix}
            1_n &0\\
            0&-1_n
        \end{bmatrix} 
    \end{equation}
    which is positive-definite on $\R^n$ and negative definite on $(\R^n)^*$. 
    Our $P_L$ and $P_R$ naturally live in the lattice $\Lambda\times\Lambda^*$
    inside this space. Under the metric $G$ this lattice is self-dual and
    even. 
    
    The full symmetry group of $\R^n\times(\R^n)^*$ preserving the metric
    is $O(n,n)$ the indefinite orthogonal group, which acts transitively
    on the set of even self-dual lattices in $\R^n\times(\R^*)^n$. Since
    every nonlinear sigma model into a torus yields an even self-dual lattice
    of signature $(n,n)$ by this construction, the space of all torus theories
    is parameterized by $O(n,n)$.
    
    There is a large subgroup of $O(n,n)$, however, that leaves the theory
    invariant. The subgroup $O(n)\times O(n)$ acting on the $P_L$ and $P_R$
    factors separately leaves the whole theory invariant. Of more interest
    is the discrete subgroup $O(n,n;\Z)$ sending the self-dual lattice to itself.
    The T-duality transformations of \cref{thm:TdualityCircle} form the
    subgroup switching a generator of $\Lambda$ with its dual in $\Lambda^*$.
    There is also the subgroup $SL(n;\Z)$ which acts to preserve both $\Lambda$
    and $\Lambda^*$ separately, corresponding to the discrete group of
    symmetries of the target space torus.
    Finally, there is the subgroup of \textit{B-field transformations}
    which shift the background B-field by an integral form.

    \begin{remark}
        T-duality can be performed in the more general context of a torus
        fibration over a base. In this case, the fibers are dualized
        and the resulting theories remain isomorphic. The famous conjecture
        by Strominger, Yau, and Zaslow of SYZ mirror symmetry \cite{Strominger:1996it}
        roughly asserts that, allowing for torus fibrations with singular fibers,
        the phenomenon of mirror symmetry can be explained using T-duality
        in this way.
    \end{remark}
    
    %---

    \subsubsection{Geometric Realizations of T-duality: Double Field Theory}%---

    Let $X=T^n$ be an $n$-dimensional torus, and consider the nonlinear sigma
    model with target $X$ as in \cref{SSS:TDualityTori}.
    Our guiding principle is that \textit{symmetries of the worldsheet theory
    correspond to gluing data on the target space}. However, the new symmetry
    we found in \cref{SSS:TDualityTori} is distinctly non-geometric in flavor.
    That is, the symmetry group $O(n,n;\Z)$ does not act on the theory as any subgroup
    of $\Diff(\R^n)$, and cannot furnish gluing data in the usual way.

    One solution to this, originally due to
    \cite{Siegel:1993th}\cite{Siegel:1993xq} and
    formalized in \cite{Hull:2004in} \cite{Hull:2006va}
    is to consider a new type of geometry known as a ``T-fold''. Generally,
    a T-fold is a manifold in which gluing data can be taken to lie in
    the T-duality group as well as the standard diffeomorphism group. Formally,
    this means a T-fold is a collection of open sets with a torus fibration over
    them, and gluing data between the open sets which is allowed to lie
    in $O(n,n;\Z)$ as well.

    \begin{example}[\protect{c.f.\ \cite{Hull:2009sg}}]
        \label{ex:DFTTwists}
        As discussed in \cref{SSS:descentToManifold}, a well-defined
        worldsheet theory can be constructed locally and glued using
        symmetries of the Lagrangian. Consider a theory which is locally
        a nonlinear sigma model with target space a (trivial) $T^n$-fibration
        over the unit interval. The interval can be glued at the endpoints
        to form the base into an $S^1$, and now the fibers must be glued by
        a symmetry of the theory. As observed in \cref{SSS:TDualityTori},
        all symmetries of the theory occur as elements of $O(n,n)$.
        A general element of the Lie algebra $\mf{o}(n,n)$ can be written
        in block-diagonal form
        \begin{equation}
            N = 
            \begin{bmatrix}
                f &Q\\
                K &-f^T
            \end{bmatrix}
        \end{equation}
        and gluing data (corresponding to monodromy data around the base)
    is given by the $O(n,n)$ element $\exp(N)$ so long as $\exp(N)\in
    O(n,n;\Z)$. The three independent components of this matrix have geometric
    interpretations:
        \begin{itemize}
            \item The matrix $f\in \mf{sl}(n;\Z)$ corresponds to an element
                of the mapping class group of $T^n$, and thus induces a large
                diffeomorphism of $T^n$. The presence of $-f^T$ indicates that
                the induced diffeomorphism of the dual torus is simultaneously
                performed. In this case, the theory glues to a nonlinear sigma
                model with target the twisted $T^n$-bundle over $S^1$ defined
                by $\exp(f)$. 
            \item The matrix $K$ induces B-field transformations of the theory.
                Thus, the target space is the trivial $T^n$-bundle over
                the base $S^1$, but the B-field does not have a global description.
                Equivalently, the dual $T^n$-bundle cannot be globally defined.
            \item The matrix $Q$, called the \textit{T-fold flux}, induces
                a $T$-duality transformation on some of the circle factors.
                In this case, the target space geometry is a T-fold, and has no
                global description as a manifold.
        \end{itemize}
    \end{example}

    There is a geometric model for T-folds known as doubling spaces, and
    the corresponding string theory with T-fold target is known as
    \textit{double field theory}. 
    \begin{remark}
        In double field theory, redundant coordinates are added to the
        action which must be eliminated by hand using the \textit{section
        constraint}. This constraint does not arise as an equation of motion
        from the Lagrangian, and because of this the theory is much harder
        to quantize. Self-dual Yang-Mills theory also encounters a similar
        problem in imposing the self-duality conditions on the field strength
        tensor \cite{Belov:2006jd}. For this reason, we focus on the
        target space geometric aspects of the theory and not the worldsheet
        theory.
    \end{remark}

    Although the specifics of the worldsheet
    theory lie outside the scope of this paper, the geometric construction
    of the target space is the motivation for our main construction. With that
    in mind, we review a particular example of a doubling space built in \cite{Hull:2009sg}.

    \begin{example}
        Consider the three-dimensional (real) manifold which is a $T^2$-fibration
        over $S^1$ with monodromy given by the Dehn twist
        \begin{equation}\label{eq:nilfoldMonodromy}
            \begin{aligned}
                f &=
                \begin{bmatrix}
                    0&0\\
                    -m&0
                \end{bmatrix}
                \in \mf{sl}_2(\Z)\\
                e^f &=
                \begin{bmatrix}
                    1&0\\
                    -m&1
                \end{bmatrix}
                \in SL_2(\Z)
            \end{aligned}
        \end{equation}.
        In the language of \cref{ex:DFTTwists}, this corresponds to setting $f$
        as above with $K$ and $Q$ zero.
        
        Let us coordinatize this manifold in the following way. Let $x = x+1$
        be a periodic coordinate parameterizing the base $S^1$, and let
        $y=y+1,z=z+1$ be periodic coordinates on the $T^2$ fibers for which the circles
        $y=0,z=0$ form a basis for the fundamental group.
        Then, the monodromy action \eqref{eq:nilfoldMonodromy} acts as
        \begin{equation}\
            \begin{aligned}
                x &\mapsto x+1\\
                y&\mapsto y+mz\\
                z&\mapsto z
            \end{aligned}
        \end{equation}
        and thus the one-forms
        \begin{equation}\
            \begin{aligned}
                P^x &= dx\\
                P^y &= dy-mxdx\\
                P^z &= dz
            \end{aligned}
        \end{equation}
        are globally defined. 
        We then define the natural metric on this space as $g= P^xP^x+P^yP^y+P^zP^z$,
        which in matrix form is
        \begin{equation}\
            g =
            \begin{bmatrix}
                1&0&0\\
                0&1&-mx\\
                0&-mx&1+m^2x^2
            \end{bmatrix}
        \end{equation}.
        
        The key idea of double field theory, which is suggested on the worldsheet
        via the existence of $\hat{X}_0$ in \eqref{eq:classicalTDualString},
        is to consider the larger space which is a $T^2\times\hat{T}^2$-fibration
        over the base. That is, the fibers over the base $S^1$ are now diffeomorphic
        to $T^4$ with coordinates $(y,z,\tilde{y},\tilde{z})$ and the induced
        monodromy is, for $x\mapsto x+1$,
        \begin{equation}\
            \begin{bmatrix}
                y\\
                z\\
                \tilde{y}\\
                \tilde{z}\\
            \end{bmatrix}
            \mapsto
            \begin{bmatrix}
                1&m&0&0\\
                0&1&0&0\\
                0&0&1&0\\
                0&0&-m&1\\
            \end{bmatrix}
            \begin{bmatrix}
                y\\
                z\\
                \tilde{y}\\
                \tilde{z}\\
            \end{bmatrix}
        \end{equation}.
        The metric on the original $T^2$-fibration extends to a metric on
        the doubling space by setting the metric on the fibers to be
        \begin{equation}
            \label{eq:genMetricDFT}
            \mc{H} =
            \begin{bmatrix}
                g&0\\
                0&g^{-1}
            \end{bmatrix}
            =
            \begin{bmatrix}
                1& -mx&0&0\\
                -mx&1+m^2x^2 &0&0\\
                0&0&1+m^2x^2 &mx\\
                0&0&mx&1\\
            \end{bmatrix}
        \end{equation}.
        Alongside this metric there is a canonical constant metric $L$ which has
        signature $(2,2)$ on the fibers. In coordinates $(y,z,\tilde{y},\tilde{z})$:
        \begin{equation}
            L =
            \begin{bmatrix}
                0& 0&1&0\\
                0&0&0&1\\
                1&0&0&0\\
                0&1&0&0
            \end{bmatrix}
        \end{equation}.
        Notice that the original $T^2$ and the dual $T^2$ are both maximal isotropic
        submanifolds under $L$.
        
        This doubled manifold can also be viewed as the quotient of a five-dimensional
        Lie group by a discrete subgroup. Namely, consider the group
        $\mc{G}\subseteq GL_5(\R)$
        whose general element is of the form
        \begin{equation}\
            g(x,y,z,\tilde{y},\tilde{z}) =
            \begin{bmatrix}
                1&mx &0&0&y\\
                0&1&0&0&z\\
                0&0&1&0&\tilde{y}\\
                0&0&-mx&1&\tilde{z}\\
                0&0&0&0&1
            \end{bmatrix}
        \end{equation}
        with discrete subgroup given by its intersection with $GL_5(\Z)$, acting
        on the left. Explicitly the integral element
        $g(\alpha,\beta,\gamma,\tilde{\beta},\tilde{\gamma})$ acts on the coordinates
        as
        \begin{equation}
            \label{eq:globalMonodromyDFT}
            \begin{aligned}
                x&\mapsto x+\alpha\\
                y&\mapsto y+m\alpha z+\beta\\
                z&\mapsto z+\gamma\\
                \tilde{y}&\mapsto \tilde{y}+\tilde{\beta}\\
                \tilde{z}&\mapsto\tilde{z}-m\alpha\tilde{y}+\tilde{\gamma}
            \end{aligned}
        \end{equation}.

        To recover the original target space, a projection on the fibers must
        be specified. Thus, we define
        \begin{definition}
            A \textit{polarization} on a doubled space $\mathbb{X}$ is a projection
            operator
            \begin{equation}
                \begin{aligned}
                    \Pi&\in\End(T\mathbb{X})\\
                    \Pi^2 &= \Pi
                \end{aligned}
            \end{equation}
            onto a maximally isotropic subbundle of $T\mathbb{X}$. If
            the distribution defined by $\Pi$ is integrable, then the submanifold
            defined by $\Pi$ is said to be the \textit{physical space-time}.
        \end{definition}
        
        Selecting the projector $\Pi_G$ to project onto the $(y,z)$ torus recovers
        the original twisted $T^2$-fibration. However, alternate choices can be
        made. Consider the projector $\Pi_H$ defined by projection onto the
        $(\tilde{y},z)$ subspace. By the global description of the monodromy action
        in \eqref{eq:globalMonodromyDFT}, this projector is well-defined on the full
        $T^4$-bundle over $S^1$. Let us now consider rewriting everything in terms
        of the new splitting of the coordinates into $(\tilde{y},z)$ and $(y,\tilde{z})$.
        The metric \eqref{eq:genMetricDFT} in these coordinates is
        \begin{equation}
            \mc{H} =
            \begin{bmatrix}
                1+m^2x^2 &0&0&mx\\
                0& 1+m^2x^2 &-mx&0\\
                0&-mx&1&0\\
                mx&0&0&1
            \end{bmatrix}
        \end{equation}
        which is not of the form
        \begin{equation}
            \begin{bmatrix}
                g&0\\
                0&g^{-1}
            \end{bmatrix}
        \end{equation}
        and thus is not induced by a standard metric on the $T^2$-fibration.
        However, this metric is induced by a pair of a metric
        \begin{equation}
            g =
            \begin{bmatrix}
                1&0\\
                0&1
            \end{bmatrix}
        \end{equation}
        and a B-field
        \begin{equation}
            B = 
            \begin{bmatrix}
                0 &mx\\
                -mx &0
            \end{bmatrix}
        \end{equation}
        on the physical $T^2$ defined by the new projection. Such a target space
        is known as a \textit{$T^3$ with $H$-flux}, and is a nontrivial example
        of a target space with a B-field. In the language of \cref{ex:DFTTwists},
        this corresponds to a nonzero $K$ with zero $f$ and $Q$.
        
        Of interest to us, however, is the alternate choice of projector
        $\Pi_T$ defined by projection to the $(y,\tilde{z})$ subspace.
        Reference to the monodromy \eqref{eq:globalMonodromyDFT} shows that
        this projector is only locally defined, but does not glue to a global
        operator. Thus, this worldsheet theory, although globally defined on the
        worldsheet, lacks an interpretation as a nonlinear sigma model. Such a theory
        is sometimes called a \textit{nonlinear sigma model with a non-geometric
        background}.
        In the language of \cref{ex:DFTTwists}, this corresponds to a nonzero $Q$
        with zero $K$ and $f$.
    \end{example}

    The previous example illustrates that the family of worldsheet theories
    has models with no geometric interpretation as a nonlinear sigma model.
    However, by doubling the torus factors a geometric realization of certain
    non-geometric sectors can be realized. 
    We will formalize certain aspects of this construction mathematically
    in \cref{chapter:main}.

    %---

\subsection{Supersymmetry and the Superstring}
    \label{SSE:superstring}

\subsubsection{Supersymmetric Strings} %---
\label{SSS:supersymmetricStrings}

    Let $\Sigma = \R\times S^1$ again be the worldsheet, coordinatized by
    $(\tau,\sigma)$ with  $\sigma =\sigma + 2\pi$. Instead of considering
    the theory of a scalar function on $\Sigma$, as in \cref{SSE:bosonicString},
    we now add to the theory a complex spinor field. 
    Scalar functions (or, more generally, sections of bundles in the
    vector representations of $\Spin(1,1)$) are referred to as \textit{bosons},
    while sections of spinor bundles are referred to as \textit{fermions}.

    The starting point is the supersymmetric Lagrangian
        \begin{equation}
            S\left[ X^\mu,\psi_-^\mu,{\psi}_+^\mu \right] 
            = \frac{1}{4\pi}\int d^2\sigma \left\{ \frac{1}{\alpha'}\partial
                X^\mu\pbar X_\mu  + \psi_+^\mu \pbar\psi_{+\mu} +
            {\psi}_-^\mu\partial{\psi}_{-\mu}\right\}
        \end{equation}
    where we have used the symmetries of $\Diff_0(\Sigma)$ and the Weyl
    group to bring the metric $\gamma$ into standard form
    \begin{equation}
        \gamma = 
        \begin{bmatrix}
            1&0\\
            0&-1
        \end{bmatrix}
    \end{equation}.
    The $\psi_-^\mu$ are right-moving fermionic fields, and ${\psi}_+^\mu$ are
    their left-moving counterparts. 

    This Lagrangian enjoys all the symmetries of the bosonic string,
    as well as additional symmetries of the fermionic fields coming from
    the $\Spin(1,1)$-representation. However, there is a new type of symmetry
    this model enjoys which is qualitatively different than the others.
    Consider the symmetry given by
    \begin{equation}
        \begin{aligned}
            X^\mu&\mapsto X^\mu - \epsilon_-\psi_+^\mu + \epsilon_+\psi_-^\mu\\
            \psi_+^\mu&\mapsto\psi_+^\mu + 2i\bar{\epsilon}_-\partial X^\mu
                -2i\bar{\epsilon}_+\pbar X^\mu
        \end{aligned}
    \end{equation}
    where $\epsilon_{\pm}$ are sections of the inverse spinor bundles parameterizing 
    the transformation.
    This symmetry mixes bosons and fermions, and symmetries of this type
    are called \textit{supersymmetries}. This theory admitted one supersymmetry
    dimension for each chirality via $\epsilon_+$ and $\epsilon_-$, and so this
    model is said to be $\ms{N}=(1,1)$ supersymmetric.

    \begin{remark}
        Following the guiding principle of \cref{SSS:descentToManifold},
        one might expect these symmetries to correspond to a certain type
        of manifold for which the symmetries are gluing data. This perspective
        on supersymmetric sigma models leads to the \textit{superspace formalism},
        where the target space is a new type of geometric object called a
        \textit{supermanifold}. The details of this lie outside the scope
        of this paper, but can be found in e.g.\ \cite{Deligne:1999qp}.
    \end{remark}
    
    As before, this theory glues to a nonlinear sigma model into an arbitrary
    Riemannian manifold. It is shown in \cite{Zumino:1979et}
    and later generalized in \cite{Alvarez-Gaume:1981exv} that additional
    supersymmetries can be added with additional structure on the target manifold.
    The theory can be upgraded to an $\ms{N}=(2,2)$ theory, where 
    there are two linearly independent transformations of each chirality, provided
    there exists a $(1,1)$ tensor $J$ satisfying
    \begin{equation}
        \begin{aligned}
            J^2&=-1\\
            \text{Nij}(J) &= 0\\
        \end{aligned}
    \end{equation}
    where $\text{Nij}$ is the Nijenhuis tensor. This operator must also be
    parallel with respect to the metric. Such an operator defines an integrable
    almost complex structure on $X$, and the parallel condition enforces that
    the manifold is K\"ahler. 

\subsubsection{The Nonlinear Sigma Model} %---

Let $(X,g)$ now be a K\"ahler manifold. Denote by capital letters $I,J$, etc.\ 
indices for the underlying real coordinates on $X$, and by lowercase letters
$i,j$, etc.\ and $\bar{i},\bar{j}$, etc.\ indices for the holomorphic (resp.\ 
antiholomorphic) coordinates. Let $D$ be the spin connection on $\Sigma$, and
$R$ the Riemann tensor on $X$. Finally, fix square roots of the canonical bundle
on $\Sigma$ to serve as spinor bundles.
With this data, the supersymmetric nonlinear $\Sigma$-model on $X$ is given by
    \begin{equation}
        \begin{aligned}
            S\left[ \Sigma, f, \psi \right] &= \int_\Sigma \big(
                \frac{1}{2}(g_{IJ})\partial_zx^I\partial_zx^J\\
                &+ \frac{i}{2}g_{i\bar{i}}\psi_-^{\bar{i}} D_z\psi_-^i
                + \frac{i}{2}g_{i\bar{i}}\psi_+^{\bar{i}} D_z\psi_+^i
                \\
            &+(R_{i\bar{i}j\bar{j}}\psi_+^i\psi_+^{\bar{i}}\psi_-^j\psi_-^{\bar{j}})
                (idz\wedge d\bar{z})
            \big)
        \end{aligned}
    \end{equation}
    where $g_{IJ}$ is the Riemannian metric, and $g_{i\bar{i}}$ is the associated
    Hermitian metric.
The fields are taking values in
\begin{center}
    \begin{tabular}{|l|l|}
        \hline
        \textbf{Field} &\textbf{Bundle}\\
        \hline
        $\psi_+^i$ & $\sqrt{K}\otimes f^*(T_X^{(1,0)})$\\
        $\psi_+^{\bar{i}}$ & $\sqrt{K}\otimes f^*(T_X^{(0,1)})$\\
        $\psi_-^i$ & $\overline{\sqrt{K}}\otimes f^*(T_X^{(1,0)})$\\
        $\psi_-^{\bar{i}}$ & $\overline{\sqrt{K}}\otimes f^*(T_X^{(0,1)})$\\
        \hline
    \end{tabular}
\end{center}.
Here $f$ is generally a map from a Riemann surface $\Sigma$ in to $X$ with
components $x^I$.
Note that the first term is the familiar $f^*g$ term from the bosonic theory.
The last two terms ensure conformal invariance and $\ms{N}=(2,2)$ supersymmetry.

\subsubsection{Topological Twisting} %---
\label{SSS:topologicalTwisting}

In \cite{witten1991mirror}, Witten proposes a modification of the supersymmetric
nonlinear sigma model which has remarkably nice properties. These theories
are called the \textit{topological $A$ and $B$ models}, and they are formed by
tensoring the fermion bundles on the worldsheet with cleverly chosen line bundles.
We review the construction now.

The nonlinear sigma model can be twisted in two ways by tensoring the fermion
bundles with $\sqrt{K}$, its conjugate, and its dual, yielding the $A$ and $B$
model twisted string theories. The $A$ model twists the fermions to make them
take values in
\begin{center}
    \begin{tabular}{|l|l|}
        \hline
        \textbf{Field} &\textbf{Bundle}\\
        \hline
        $\psi_+^i$ & $f^*(T_X^{(1,0)})$\\
        $\psi_+^{\bar{i}}$ & $K\otimes f^*(T_X^{(0,1)})$\\
        $\psi_-^i$ & $\overline{K}\otimes f^*(T_X^{(1,0)})$\\
        $\psi_-^{\bar{i}}$ & $ f^*(T_X^{(0,1)})$\\
        \hline
    \end{tabular}
\end{center}
while the $B$ model takes values in
\begin{center}
    \begin{tabular}{|l|l|}
        \hline
        \textbf{Field} &\textbf{Bundle}\\
        \hline
        $\psi_+^i$ & $K\otimes f^*(T_X^{(1,0)})$\\
        $\psi_+^{\bar{i}}$ & $f^*(T_X^{(0,1)})$\\
        $\psi_-^i$ & $\overline{K}\otimes f^*(T_X^{(1,0)})$\\
        $\psi_-^{\bar{i}}$ & $ f^*(T_X^{(0,1)})$\\
        \hline
    \end{tabular}
\end{center}.
Either twist destroys half the supersymmetry, and the resulting twisted theory
only has two supercharges instead of four. 

We can examine the $B$ model in more detail: the field 
    \begin{equation}
            \eta^{\bar{i}} =
        \psi_+^{\bar{i}} + \psi_-^{\bar{i}}
    \end{equation}
is a section of $f^*T_X^{(0,1)}$,
the field
    \begin{equation}
        \rho^i = \psi^i_+ + \psi^i_-
    \end{equation}
is a section of $f^*\Omega_X$ (with $(1,0)$ and $(0,1)$ components split as above), and 
    \begin{equation}
        \theta_i = g_{i\bar{i}}\left(  \psi_+^{\bar{i}} - \psi_-^{\bar{i}}\right)
    \end{equation}
is a section of $f^*\Omega^{(1,0)}_X$. Taking products of these allows us to
associate to each $(0,q)$ form with values in $\bigwedge^pT^{(1,0)}_X$ a local
operator built from the fields. Namely, for $\theta\in
\Omega^q(X,\bigwedge^pT^{(1,0)}_X)$ given locally as
    \begin{equation}
        \theta = h_{\bi_1\dots\bi_q}^{j_1\dots j_p}
        d\bar{z}^{\bi_1}\wedge\dots\wedge d\bar{z}^{\bi_q}\otimes
        \partial_{j_1}\wedge\dots\wedge \partial_{j_p}
    \end{equation}
we can associate the operator $\oh_\theta$ defined locally as
    \begin{equation}
        \oh_\theta = h_{\bi_1\dots\bi_q}^{j_1\dots
        j_p}\eta^{\bi_1}\dots\eta^{\bi_q}\theta_{j_1}\dots\theta_{j_p}
    \end{equation}

The surviving supersymmetry $Q$ is given by
    \begin{equation}
        \label{eq:bmodelSymmetry}
        \begin{aligned}
            \delta x^i &= 0\\
            \delta x^{\bi} &= i\alpha\eta^{\bi}\\
            \delta \psi_+^i &= -\alpha\partial x^i\\
            \delta \psi_+^{\bi} &=
            -i\alpha\psi_-^{\bar{j}}\Gamma^{\bi}_{\bar{j}\bar{k}}\psi_+^{\bar{k}}\\
            \delta \psi^i_- &= -\alpha\pbar x^i\\
            \delta\psi_-^{\bi} &=
            -i\alpha\psi_{+}^{\bar{j}}\Gamma^{\bi}_{\bar{j}\bar{k}}\psi_-^{\bar{k}}
        \end{aligned}
    \end{equation}
where $\alpha$ parameterizes the deformation as a
section of $\sqrt{K}^{-1}$. To say that $Q$ is given by the
above means that $Q$ generates the deformation in the sense that
    \begin{equation}
        \delta W = -i\left\{ Q(\alpha), W \right\}
    \end{equation}
for any operator $W$.

The operator $Q$ satisfies $Q^2 = 0$ up to the equations of motion for the
fields,
and furthermore
    \begin{equation}
        \delta\oh_{\theta} = \left\{ Q,\oh_{\theta} \right\} = \oh_{\pbar \theta}
    \end{equation}.
    Thus, $Q$ acts on the space of fields, which are sections
    of $\Omega^q(X,\bigwedge^p T_X^{(1,0)})$, as the $\pbar$ operator.

In \cite{witten1991mirror}, Witten shows that the topologically twisted theories
localize, in the sense that the string states in the topological theory
are computed by $Q$-cohomology.
\begin{theorem}[\protect{\cite{witten1991mirror}*{Section~4.2}}]
    In the topological $B$-model, the string spectrum is given by the
    Dolbeault cohomology of $\Omega^q(X,\bigwedge^pT_X^{1,0})$.
\end{theorem}

\subsubsection{Boundary Conditions for Open Strings} %---
\label{SSS:openStringBoundary}

Let us continue focusing on the twisted B-model.
When the worldsheet $\Sigma = \R\times [0,\pi]$ is the worldsheet of an open
string, additional terms can be added to the action which are integrals
of forms supported on the boundary $\partial\Sigma$. When the
target space $X$ supports a vector bundle $E$ with connection,
the \textit{Chan-Paton term} can be added to couple the ends of the string
to the vector bundle.
This amounts to adding a term to the action
    \begin{equation}
        \label{eq:boundaryActionChanPaton}
        S_{\partial\Sigma} = \oint_{\partial\Sigma}\left( f^*(A) -
        i\eta^{\bar{i}}F_{\bar{i}j}\rho^j \right)
    \end{equation}
where $A$ is the connection $1$-form for the vector bundle $E$, and $F$ is its
field strength.
The operators of the theory now take values in 
$\Omega^q(X,\bigwedge^pT^{(1,0)}_X)\otimes \shfEnd(E)$. However, 
constraints on the endpoint involving the field strength $F$ fix the values 
of $\theta$ in terms of $\eta$:
\begin{equation}
    \theta_j = F_{j\bar{k}}\eta^{\bar{k}}
\end{equation}
and so the surviving operators are simply $(0,q)$-forms on $X$ with values
in $\End(E)$.

This boundary term is invariant under $Q$ only for certain values of $A$.
By varying \eqref{eq:boundaryActionChanPaton} with the transformations
in \eqref{eq:bmodelSymmetry}, the change in the action is zero when $E$
is a holomorphic vector bundle and the connection is compatible 
with the holomorphic structure. The supersymmetry $Q$ acts on
operators by the dbar operator $\pbar_E$ on $E$, and the spectrum
is again computed using cohomology.

\begin{proposition}
    The string spectrum of the open string B-model with ambient vector
    bundle $E$ is given by bundle-valued Dolbeault cohomology
    of $\Omega^q(X)\otimes\shfEnd(E)$.
\end{proposition}
\begin{remark}
    In the topological A-model, a different supersymmetry is preserved. In
    order to preserve this supersymmetry, the endpoints no longer couple
    to holomorphic vector bundles but instead couple to \textit{coisotropic
    branes} \cite{Kapustin:2001ij}. The open string spectrum has not been fully
    defined in the A-model, but in the case of the endpoints coupling to
    Lagrangian submanifolds it is believed that the open string spectrum is
    computed via Lagrangian Intersection Floer Theory, defined in
    \cite{Fukaya2009Lagrangian}.
\end{remark}

In general, the open string can couple to many more objects than just holomorphic
vector bundles. By allowing the endpoints of a string to couple to
different objects on each end, the theory has the structure of a category
whose objects are boundary conditions strings can couple to, and whose $\Hom$-spaces
are given by the open string spectra with prescribed boundary conditions on each
end. These objects are called \textit{D-branes}.

In \cite{Ooguri:1996ck} a geometric description for D-branes preserving
the supersymmetry is outlined.
\begin{theorem}[\protect{\cite{Ooguri:1996ck}*{Section~2.3}}]
    The D-branes in a Calabi-Yau target space preserving B-type supersymmetry
    compatible with the B-model twist are exactly the holomorphic submanifolds
    supporting holomorphic vector bundles.
\end{theorem}
\begin{remark}
    It is commonly accepted that D-branes should be parameterized by the
    \textit{coherent sheaves} on the target space, not just the vector bundles
    on submanifolds of the target. This is somewhat justified in
    \cite{aspinwall2001derived} where it is shown that certain deformations
    of the conformal field theory allow for D-branes to arise as cokernels
    of sheaf morphisms between the associated locally free sheaves.
    This leap from the category of holomorphic vector bundles to coherent
    sheaves is discussed in \cite{Aspinwall:2009isa}*{Section~5.3.3}.
\end{remark}

In \cite{katz2002d} the
B-model spectrum was computed between two such D-branes:
\begin{theorem}[\protect{\cite{katz2002d}}]
    \label{thm:katzSharpeExtGroups}
    Let $\mc{E}$ and $\mc{F}$ be two coherent sheaves on a Calabi-Yau manifold
    $X$. Then, the topological B-model open string spectrum from $\mc{E}$ to $\mc{F}$
    is given by the $\Ext$-groups
    \begin{equation}
        \Hom_B(\mc{E},\mc{F}) =\bigoplus^q \Ext^q(\mc{E},\mc{F})
    \end{equation}
\end{theorem}

This was refined in \cite{aspinwall2001derived}, where the total category of such
D-branes was computed: 
\begin{theorem}[\protect{\cite{aspinwall2001derived}*{Section~2.7}}]
    The category of D-branes in the topological B-model on a Calabi-Yau
    manifold $X$ is equivalent to the derived category of coherent sheaves on $X$.
\end{theorem}

%=== Chapter Mathematics 
\section{Mathematics Background}
    \label{chapter:math}
    We now review the prerequisite mathematics for the main construction
    of the paper. We concern ourselves primarily with rank one,
    space filling B-type D-branes, as defined in \cref{SSS:openStringBoundary},
    and we set our target space to be a complex torus.
    In this chapter, we will develop the theory of line bundles with connections,
    focusing in particular on the construction of \textit{differential characters}
    by Cheeger and Simons, defined in \cite{simons1985differential}. We
    specialize these results to line bundles over complex tori, and review the
    computation of cohomology for these line bundles. Finally, we review the
    mathematics of \textit{generalized complex geometry}, developed in
    \cite{gualtieri2004generalized}.

    For this chapter, if $X$ is a projective complex manifold we denote by
    $\mc{A}^k$ the sheaf of smooth differential $k$-forms on $X$, and by
    $\mc{A}^{p,q}$ the subsheaves of $p,q$ forms given by the Hodge decomposition.
    In particular, $\mc{A}^k(X)$ is the module of global smooth $k$-forms.
    All forms are considered to take complex values, and we denote by $\mc{A}^k_\R$,
    etc.\ the corresponding subsheaf of real-valued forms.
    The sheaf of holomorphic functions is denoted $\oh_X$, and the sheaves of 
    holomorphic $k$-forms on $X$ are denoted by $\Omega^k$.
    If $\mc{L}$ is a line bundle over $X$, then denote by $\mc{A}^k(\mc{L})$ the
    sheaf of differential $k$-forms with values in $\mc{L}$, and by
    $\mc{A}^{p,q}(\mc{L})$ the sheaf of $(p,q)$ forms with values in $\mc{L}$.
    When necessary, the module of global sections of $\mc{A}^k(\mc{L})$
    will be denoted $\mc{A}^k(X, \mc{L})$.

\subsection{Connections on Holomorphic Line Bundles}
    For this section we follow \cite{huybrechts2005complex}.
    Let $X$ be a projective complex manifold and $\mc{L}\to X$ a complex
    line bundle. 

    \subsubsection{Connections and Structure on Line Bundles} % ---
    \begin{definition}
        A \textit{(smooth) connection} on $\mc{L}$ is a $\C$-linear sheaf morphism
            \begin{equation}
                \nabla:\mc{A}^0(\mc{L})\to\mc{A}^1(\mc{L})
            \end{equation} 
        satisfying the Leibniz rule
            \begin{equation}
                \label{eq:nablaLeibniz}
                \nabla(fs) = df\otimes s + f\otimes \nabla s
            \end{equation} 
        for $f$ a local smooth function on $X$ and $s$ a local smooth section of $L$.
    \end{definition}

    \begin{proposition}[\protect{\cite{huybrechts2005complex}*{Corollary~4.2.4}}]
        The set of all smooth connections on $\mc{L}$ naturally forms an affine space over
        $\mc{A}^1(X,\End(\mc{L}))\cong\mc{A}^1(X)$ the space of global
        one-forms with values in $\End(\mc{L})\cong C^\infty(X,\C)$ i.e. global
        complex one-forms.
    \end{proposition}

    Recall that complex line bundles are simply smooth vector bundles with
    typical fiber a complex vector space, such that the transition functions
    can be chosen to be smooth functions with values in $\GL_1(\C)\cong \C^*$.
    In particular, a complex line bundle is given by the data of a $1$-cocycle
    $\psi\in H^1(X,\C^*)$.

    \begin{definition}
        \label{def:holomorphicStructure}
        A \textit{holomorphic structure} on a line bundle $\mc{L}$ is the data
        of a holomorphic $1$-cocycle $\psi\in H^1(X,\oh_X^*)$ which induces
        the defining cocycle of $\mc{L}$ in $H^1(X,\C^*)$.
    \end{definition}
    Alternatively, holomorphic line bundles are such that the projection map
    from the total space of $\mc{L}$ to the base space $X$ is a holomorphic map.

    Since the underlying manifold is complex, the hodge decomposition of forms
    allows us to decompose a connection $\nabla$ into its $(1,0)$ and $(0,1)$
    parts
        \begin{equation}
            \begin{aligned}
                \nabla &= \nabla^{1,0}+\nabla^{0,1}\\
                \nabla^{1,0} &= \Pi^{1,0}\circ\nabla:\mc{A}^0(\mc{L})\to \mc{A}^{1,0}(\mc{L})\\
                \nabla^{0,1} &= \Pi^{0,1}\circ\nabla:\mc{A}^0(\mc{L})\to \mc{A}^{0,1}(\mc{L})\\
            \end{aligned}
        \end{equation} 
    where $\Pi^{p,q}$ is the projector associated to the Hodge decomposition
        \begin{equation}
            \mc{A}^k(\mc{L}) = \bigoplus_{p+q=k}\mc{A}^{p,q}(\mc{L})
        \end{equation}.
    Although connections are in general non-canonical, the holomorphic structure
    on the base $X$ induces a natural operator on $\mc{L}$:
    \begin{proposition}
        The operator $\pbar:\mc{A}^{p,q}\to \mc{A}^{p,q+1}$ 
        extends naturally to a $\C$-linear operator 
            \begin{equation}
                \label{eq:dbarDef}
                \pbar_\mc{L}:\mc{A}^{p,q}(\mc{L})\to\mc{A}^{p,q+1}(\mc{L})
            \end{equation}.
        Moreover, this operator still satisfies the Leibniz identity:
        \begin{equation}
            \label{eq:dbarLeibniz}
            \pbar_\mc{L}(f\alpha) = \pbar(f)\wedge \alpha + f\pbar_\mc{L}(\alpha)
        \end{equation} 
        for all
        $f\in \mc{A}^0$ and $\alpha\in\mc{A}^{p,q}(\mc{L})$, and is square-zero
        \begin{equation}
            \label{eq:dbarSquareZero}
            \pbar_\mc{L}^2 = 0
        \end{equation}.
    \end{proposition}
    A $\C$-linear operator satisfying the conditions \eqref{eq:dbarLeibniz}
    and \eqref{eq:dbarSquareZero} we call a \textit{dbar} operator on $\mc{L}$. 

    The dbar operators on $\mc{L}$ parameterize the holomorphic
    structures on the underling complex line bundle.
    \begin{theorem}[\protect{\cite{huybrechts2005complex}*{Theorem~2.6.26}}]
        The association of a holomorphic line bundle $\mc{L}$ with the canonical dbar
        operator $\pbar_\mc{L}$ is a bijective correspondence between the space of
        holomorphic structures on the underlying complex line bundle and the space
        of dbar operators on $\mc{L}$ satisfying the Leibniz rule and squaring to zero.
    \end{theorem}
    
    \begin{definition}
        A connection $\nabla$ on $\mc{L}$ is \textit{compatible with the holomorphic
        structure} if
            \begin{equation}
                \nabla^{0,1} = \pbar_\mc{L}
            \end{equation}.
    \end{definition}

    \begin{remark}
        This definition should not be confused with that of a holomorphic (or
        algebraic) connection on $\mc{L}$. Indeed, a holomorphic connection on
        $\mc{L}$ is by definition a $\C$-linear sheaf homomorphism
            \begin{equation}
                D:\mc{L}\to\Omega^1_X\otimes \mc{L}
            \end{equation} 
        satisfying the holomorphic version of the Leibniz rule
            \begin{equation}
                D(fs) = \partial(f)\otimes s + fD(s)
            \end{equation}.
        This condition is far more restrictive than being compatible with
        the holomorphic structure. For example, only topologically
        trivial line bundles admit holomorphic connections, but all line bundles
        admit connections compatible with a holomorphic structure.
    \end{remark}
    
    \begin{proposition}[\protect{\cite{huybrechts2005complex}*{Corollary~4.2.13}}]
        The space of all connections on $\mc{L}$ compatible with the 
        holomorphic structure naturally forms an affine space over
    $\mc{A}^{1,0}(X,\End(\mc{L}))\cong \mc{A}^{1,0}(X))$. 
    \end{proposition}

    The final structure to consider imposing on a line bundle is the following:
    \begin{definition}
        An \textit{hermitian structure} on a complex vector bundle $E$ is the data
        of a positive-definite hermitian form on each fiber of $E$, varying smoothly.
        In particular, an hermitian structure on a complex line bundle is the 
        data of a positive smooth function on the base.
    \end{definition}

    \begin{definition}
        A connection $\nabla$ on a complex line bundle $\mc{L}$ equipped with an hermitian
        structure $h$ is said to be \textit{compatible with the hermitian structure}
        if
        \begin{equation}
            d(h(s_1,s_2)) = h(\nabla s_1,s_2)+h(s_1,\nabla s_2)
        \end{equation}
        for all (local) sections $s_1,s_2$ of $\mc{L}$.
    \end{definition}
    
\begin{proposition}[\protect{\cite{huybrechts2005complex}*{Corollary~4.2.11}}]
        Let $(\mc{L},h)$ be a complex line bundle with an hermitian structure. The
        space of connections on $\mc{L}$ compatible with the hermitian structure
        naturally forms an affine space over $\mc{A}^{1}(X, \End(\mc{L},h))$
        where $\End(\mc{L},h)$ is the subsheaf of $\End(\mc{L})$ for which local sections
        $a$ satisfy
            \begin{equation}
                h(as_1,s_2)+h(s_1,as_2)=0
            \end{equation} 
        for all $s_1,s_2$ local sections of $\mc{L}$. 

        In particular, $\mc{A}^1(X,\End(\mc{L},h))\cong \mc{A}^1_\R(X,i\R))$.
    \end{proposition}
    
\begin{theorem}[\protect{\cite{huybrechts2005complex}*{Proposition~4.2.14}}]
        For $(\mc{L},h)$ a complex line bundle with an hermitian structure and a
        choice $\pbar_\mc{L}$ of holomorphic structure on $\mc{L}$, there exists a
        unique connection $\nabla$ compatible with both the holomorphic and
        hermitian structures called the \textit{Chern connection} on $\mc{L}$.
    \end{theorem}

    If $h$ is the positive real function on $X$ defining
    the hermitian structure, then the Chern connection is given locally by
    \begin{equation}\label{chernConnectionEquation}
        \nabla = d + \partial\log h
    \end{equation}.

    Line bundles with hermitian structures have an alternative description
    in terms of their structure group $U(1)$.
    \begin{theorem}
        \label{thm:principalIsLineBundle}
        There is a one-to-one correspondence between principal $U(1)$-bundles
        over $X$ and line bundles with hermitian structures $(\mc{L},h)$. The forward
        association is given by taking the associated bundle to the standard
        representation of $U(1)$ on $\C$, with hermitian structure given
        by the standard hermitian form on $\C$. The backward direction is
        given by taking the unitary frame bundle of $\mc{L}$.
    \end{theorem}
    \begin{remark}
        There is a notion of a connection on a principal bundle which
        induces connections on associated bundles, and the correspondence
        of \cref{thm:principalIsLineBundle} extends to connections as well.
        In many ways the connections on principal bundles are more well-behaved
        than their vector bundle counterparts, but this lies outside the scope
        of this paper. We will see that the theory of differential
        characters will suffice.
    \end{remark} 

    \subsubsection{Curvature and Chern Classes}
    
    Let $\mc{L}$ be a line bundle over $X$ with connection $\nabla$.
    We have observed that $\nabla$ itself is not a one-form, but instead
    sits in an affine space over $\mc{A}^1(X)$. 
    The square of $\nabla$, however, behaves more nicely. Recall
    that $\nabla:\mc{A}^0(\mc{L})\to\mc{A}^1(\mc{L})$ can be extended to higher
    degrees by the Leibniz rule \eqref{eq:nablaLeibniz}.
    \begin{definition}
        The \textit{curvature} of $\nabla$, denoted $F_\nabla$,
        is the $\C$-linear map
        \begin{equation}
            F_\nabla := \nabla^2:\mc{A}^0(\mc{L})\to\mc{A}^2(\mc{L})
        \end{equation}.
    \end{definition} 
    \begin{proposition}[\protect{\cite{huybrechts2005complex}*{Lemma~4.3.2}}]
        The curvature is $\mc{A}^0$-linear, and thus can be thought 
        of as an element of $\mc{A}^2(X,\End(\mc{L}))\cong\mc{A}^2(X)$.
        Furthermore, $F_\nabla$ is $d$-closed.
    \end{proposition} 

    \begin{remark}[Local computations]
        In a small enough open set $U$, any connection is of the form
        \begin{equation}
            \nabla = d + A
        \end{equation}
        for some $A\in\mc{A}^1(U)$. Then,
        \begin{equation}
            F_\nabla = (d+A)^2 = dA
        \end{equation}.
        This makes it clear that 
        \begin{equation}
            dF_\nabla = d^2A = 0
        \end{equation}.
    \end{remark}

    Since $F_\nabla$ is $d$-closed, it determines a cohomology class
    in $H^2(X,\C)$. This cohomology class is independent of choice
    of connection on $\mc{L}$. Indeed, any two connections
    differ by a one-form:
    \begin{equation}
        \nabla_2-\nabla_1 = A\in\mc{A}^1(X)
    \end{equation}
    and thus the shift of the curvature
    \begin{equation}
        F_{\nabla_2} = F_{\nabla_1} + dA
    \end{equation}
    is by an exact form, which does not change the cohomology.
    Since every line bundle admits an hermitian structure and a complex
    structure, we see that the curvature $F_\nabla$ of any connection
    on $\mc{L}$ determines a cohomology class $\left[ F_\nabla \right]\in H^{1,1}(X,\R)$
    which only depends on the underlying complex line bundle. 
    We formalize this with the following:
    \begin{definition}
        \label{def:firstChern}
        The \textit{first Chern class} of $\mc{L}$, denoted $c_1(\mc{L})$,
        is the cohomology class
        \begin{equation}
            c_1(\mc{L}) = \frac{i}{2\pi}\left[ F_\nabla \right]
        \end{equation}
        where $F_\nabla$ is the curvature of any connection on $\mc{L}$.
    \end{definition} 
    \begin{definition}[Alternate, \protect{\cite{Griffiths1994principles}*{Page~141}}]
        The first Chern class is alternately described in the following way.
        Consider the exponential exact sequence
        \begin{equation}
            \begin{tikzcd}
                0\arrow[r] &\Z\arrow[r] &\oh_X\arrow[r, "\exp"] &\oh_X^*\arrow[r] &0
            \end{tikzcd} 
        \end{equation}
        given by the exponential map from the sheaf $\oh_X$ of holomorphic
        functions to the subsheaf $\oh_X^*$ of invertible holomorphic functions.
        The long exact sequence on homology induces a map
        \begin{equation}
            \delta_2: H^1(X,\oh_X^*)\to H^2(X,\Z)
        \end{equation}.
        Associating a holomorphic line bundle to its class in $H^1(X,\oh_X)$
        in the sense of \cref{def:holomorphicStructure} the map
        $\delta_2$ coincides with the first Chern class map.
        This explains the normalization factor of $\frac{i}{2\pi}$: the
        form $\frac{i}{2\pi}F_\nabla$ is an integral form.
    \end{definition} 
    
    %---

\subsection{Holonomy and Differential Characters}

    Let $(\mc{L},h)$ be a complex hermitian line bundle over $X$ a projective complex
    manifold, and let $\nabla$ be a connection on $\mc{L}$ compatible with $h$.
    We have already seen that the topological type of $\mc{L}$ is determined by the
    first Chern class as an element of $H^2(X,\Z)$, but a finer invariant is needed
    to determine the isomorphism class of $\mc{L}$ as a complex line bundle with
    $U(1)$-connection. The theory of differential cohomology, first
    appearing in \cite{simons1985differential},
    provides a natural framework for this classification, which we will review now.
    We will mainly follow \cite{Bar2014differential}.

    \subsubsection{Cheeger-Simons Differential Characters} %---

    The theory of Cheeger-Simons differential characters tracks the information
    contained in principal bundles with connection. In fact, the theory extends
    to the so-called principal $n$-bundles of higher geometry, but for present
    purposes we focus on its application to principal $U(1)$-bundles. 

    \begin{definition}
        \label{def:CheegerSimonsDiffCharacter}
        Let $X$ be a smooth manifold. Denote by $C_k(X), Z_k(X)$ and $B_k(X)$
        the abelian groups of $k$-chains, $k$-cycles, and $k$-boundaries.
        A map $f\in\Hom_\Z(Z_{k-1}(X),U(1))$ is said to be a
        \textit{Cheeger-Simons differential character} of degree $k$ if, for
        every $\Sigma\in C_k(X)$,
        \begin{equation}
            \label{curvatureExistence}
            f(\partial\Sigma) = \exp\left(2\pi i \int_\Sigma\omega\right)
        \end{equation}
        for some $\omega\in A^k(X)$.
        The form $\omega$ is uniquely determined by $f$, and is called the
        \textit{curvature} of $f$. We denote it by $\curv(f)$.

        The set of all Cheeger-Simons differential characters forms a 
        subgroup 
        \begin{equation}
            \hat{H}^k(X)\subseteq \Hom_\Z(Z_{k-1}(X),U(1))
        \end{equation}
        called the
        \textit{Cheeger-Simons differential cohomology group} of degree $k$.
    \end{definition}

    \begin{example}[$k=1$]
        In degree $1$, the Cheeger-Simons differential cohomology group
        $\hat{H}^1(X)$ is a subgroup of $\Hom_\Z(Z_0(X),U(1))$, which
        can be thought of as $\Hom(X,U(1))$, the group of (set-theoretic) maps
        from $X$ to $U(1)$. 

        Let $f\in \hat{H}^1(X)$ be such a map. Fixing a point $x_0\in X$
        and a point $x$ in a small neighborhood of $x_0$, let $\gamma$ be
        the straight line connecting $x_0$ to $x$. Then, since $f$ admits a
        curvature form \eqref{curvatureExistence}, 
        \begin{equation}
            f(x) = f(x_0)\exp\left(2\pi i \int_{\gamma}\curv(f)\right)
        \end{equation}
        showing that $f$ is, in fact, a smooth function.

        Conversely, if $f\in C^\infty(X,U(1))$ is a smooth $U(1)$-valued function
        on $X$, pick a local lift $\tilde{f}:X\to \R$ to the universal cover,
        so that locally
        \begin{equation}
            f(x) = \exp(2\pi i \tilde{f}(x))
        \end{equation}.
        Then, the one-form $d\tilde{f}$ is the curvature of $f$, showing $f\in\hat{H}^1(X)$.
    \end{example}

    \begin{example}[$k=2$]
        In degree $2$, a Cheeger-Simons differential character $f\in\hat{H}^2(X)$
        is a map
        \begin{equation}
            f:Z_1(X)\to U(1)
        \end{equation}
        satisfying \eqref{curvatureExistence}. This map determines a principal
        $U(1)$-bundle with connection, as outlined in \cite{Barrett:1991aj}.

        Conversely, suppose $P\to X$ is a principal $U(1)$-bundle with connection
        $\nabla$. Then, the holonomy mapping
        \begin{equation}
            \hol_\nabla:Z_1(X)\to U(1)
        \end{equation}
        defines a map $f_\nabla\in\Hom_\Z(Z_1(X),U(1))$. The curvature $F_\nabla$
        of the connection then functions as the curvature for $f$, up to normalization:
        \begin{equation}
            \curv(f) = \frac{i}{2\pi}F_\nabla
        \end{equation}.
    \end{example}
    
    The Cheeger-Simons differential cohomology group has four natural maps
    associated to it.
    \begin{itemize}
        \item The curvature map 
            \begin{equation}
                \curv: \hat{H}^k(X)\to\mc{A}^k(X)
            \end{equation}
            sending a character to its curvature form has already been defined.
            For any $f\in\hat{H}^k(X)$, $\curv(f)$ is closed and has integral
            periods, hence the image of $\curv$ lies in $\mc{A}^k_\Z(X)$. Thus
            we get a map
            \begin{equation}\label{DCcurvMap}
                \delta_1: \hat{H}^k(X)\to\mc{A}^k_\Z(X)
            \end{equation}.
            We call this map the \textit{curvature map}.
        \item Let $\chi\in H^{k-1}(X,\R/\Z)$, and identify it with a function
            \begin{equation}
                \chi:H_{k-1}(X,\Z)\to U(1)
            \end{equation}
            in the natural way. Then, $\chi$ in particular defines an element of
            $\hat{H}^k(X)$ with zero curvature. This association assembles into
            a group homomorphism
            \begin{equation}\label{DCflatInclusion}
                \iota_1: H^{k-1}(X,\R/\Z)\to\hat{H}^k(X)
            \end{equation}
            which we call the \textit{inclusion of flat characters}.
        \item Let $f\in\hat{H}^k(X)$ and pick a lift of $f$ to the universal
            cover $\R$ of $U(1)$:
            \begin{equation}
                \tilde{f}:Z_{k-1}(X)\to \R
            \end{equation}
            such that $f(\gamma) = \exp(2\pi i \tilde{f}(\gamma))$. 
            Now, consider the cochain
            \begin{equation}
                \begin{aligned}
                    \delta_2(f) &: C_{k}(X)\to \Z\\
                    \Sigma &\mapsto \int_\Sigma \curv(f) - \tilde{f}(\partial\Sigma)
                \end{aligned}
            \end{equation}.
            This is a cocycle, and hence defines a cohomology class
            $\delta_2(f)\in H^k(X,\Z)$. The notation is justified by noticing that
            the cohomology class does not depend on choice of lift. 
            This assembles into a well-defined group homomorphism
            \begin{equation}\label{DCchernMap}
                \delta_2:\hat{H}^k(X) \to H^k(X,\Z)
            \end{equation}
            which we call the \textit{Chern character map}.
        \item Let $\omega\in \mc{A}^{k-1}(X)$ be a $k-1$-form. Integration
            against cycles induces a map
            \begin{equation}
                \begin{aligned}
                    \int\omega&: Z_{k-1}(X)\to \R\\
                \end{aligned}
            \end{equation}
            which exponentiates to
            \begin{equation}
                \begin{aligned}
                    \iota_2(\omega)&: Z_{k-1}(X)\to U(1)\\
                    \gamma &\mapsto \exp\left( 2\pi i \int_\gamma \omega \right)
                \end{aligned}
            \end{equation}.
            By Stokes' theorem, this map satisfies
            \begin{equation}
                \iota_2(\omega)(\partial\Sigma) = \exp\left( 2\pi i \int_\Sigma d\omega \right)
            \end{equation}
            expressing $\iota_2(\omega)$ as a Cheeger-Simons differential character.

            If $\omega$ is closed and has integral periods, then $\iota_2(\omega)$
            is the trivial map. Thus, $\iota_2$ builds a well-defined map
            \begin{equation}\label{DCtopTrivialMap}
                \iota_2: \mc{A}^{k-1}(X)/\mc{A}^{k-1}_\Z(X) \to \hat{H}^k(X)
            \end{equation}
            which we call the \textit{topological trivialization map}.
    \end{itemize}

    These four maps allow the Cheeger-Simons differential cohomology group
    to fit into the differential character diagram:
    \begin{equation}
        \label{diffCharacterDiagram}
        % https://q.uiver.app/#q=WzAsMTUsWzEsMCwiMCJdLFsxLDIsIkhee2stMX0oXFxSKVxcaHNwYWNley0xZW19Il0sWzEsNCwiMCJdLFszLDIsIlxcaGF0e0h9XmsiXSxbMiwxLCJIXntrLTF9KFxcUi9cXFopXFxoc3BhY2V7LTNlbX0iXSxbMiwzLCJcXGZyYWN7QV57ay0xfX17QV57ay0xfV9cXFp9Il0sWzQsMSwiSF5rKFxcWikiXSxbNSwyLCJIXmsoXFxSKVxcaHNwYWNley0xZW19Il0sWzQsMywiQV5rX1xcWiJdLFs1LDQsIjAiXSxbNSwwLCIwIl0sWzAsMSwiXFxidWxsZXQiXSxbMCwzLCJcXGJ1bGxldCJdLFs2LDEsIlxcYnVsbGV0Il0sWzYsMywiXFxidWxsZXQiXSxbMSw0LCJyIl0sWzQsNiwiXFxoc3BhY2V7M2VtfVxcZGVsdGEiLDAseyJzaG9ydGVuIjp7InNvdXJjZSI6NjB9fV0sWzYsNywiaSJdLFsxLDUsImRlUiIsMl0sWzgsNywiZGVSIiwyXSxbNSw4LCJkIiwyXSxbMiw1XSxbNiwxMF0sWzAsNF0sWzgsOV0sWzQsMywiXFxpb3RhXzEiXSxbMyw4LCJcXGRlbHRhXzEiXSxbNSwzLCJcXGlvdGFfMiJdLFszLDYsIlxcZGVsdGFfMiJdLFsxMSwxXSxbMTIsMV0sWzcsMTNdLFs3LDE0XV0=
        \begin{tikzcd}[column sep=small]
            & 0 &&&& 0 \\
            H^{k-1}(\Z)\hspace{-1em} && {H^{k-1}(\R/\Z)\hspace{-2em}} && {H^k(\Z)} && H^k(\R/\Z) \\
            & {H^{k-1}(\R)\hspace{0em}} && {\hat{H}^k} && {H^k(\R)\hspace{-1em}} \\
            \mc{A}^{k-1}_\Z\hspace{-1em} && {\frac{\mc{A}^{k-1}}{\mc{A}^{k-1}_\Z}} && 
            {\mc{A}^k_\Z} && {\frac{\mc{A}^k}{\mc{A}^k_\Z}} \\
            & 0 &&&& 0
            \arrow[from=1-2, to=2-3]
            \arrow[from=2-1, to=3-2]
            \arrow["{\hspace{3em}\delta}", shorten <=32pt, from=2-3, to=2-5]
            \arrow["{\iota_1}", from=2-3, to=3-4]
            \arrow[from=2-5, to=1-6]
            \arrow["i", from=2-5, to=3-6]
            \arrow["r", from=3-2, to=2-3]
            \arrow["deR"', from=3-2, to=4-3]
            \arrow["{\delta_2}", from=3-4, to=2-5]
            \arrow["{\delta_1}", from=3-4, to=4-5]
            \arrow[from=3-6, to=2-7]
            \arrow[from=3-6, to=4-7]
            \arrow[from=4-1, to=3-2]
            \arrow["{\iota_2}", from=4-3, to=3-4]
            \arrow["d"', from=4-3, to=4-5]
            \arrow["deR"', from=4-5, to=3-6]
            \arrow[from=4-5, to=5-6]
            \arrow[from=5-2, to=4-3]
        \end{tikzcd}
    \end{equation}
    where each diagonal is an exact sequence. Here, $r,\delta,i$ form the
    long exact sequence in cohomology associated to the short exact sequence
    \begin{equation}
        \begin{tikzcd}
            0\arrow[r] &\Z\arrow[r] &\R\arrow[r] &\R/\Z\arrow[r] &0
        \end{tikzcd}
    \end{equation}, $d$ is the exterior derivative, and $deR$ is the induced map from
    the De Rham theorem.

    %---

    \subsubsection{ Hermitian Line Bundles and Differential Characters} %---
    
    The elements of $\hat{H}^2(X)$ are in bijective correspondence with
    isomorphism classes of principal $U(1)$-bundles over $X$ with a connection,
    and hence are in bijection with isomorphism classes of complex line bundles
    over $X$ with an hermitian structure and compatible connection. In this
    context, the four fundamental morphisms of a differential character are
    well-known.

    Let $(\mc{L},h)$ be a complex line bundle with an hermitian structure over $X$,
    and $\nabla$ an hermitian connection on $\mc{L}$.
    Then,
    \begin{itemize}
        \item $\delta_1(\mc{L},\nabla) = \curv(\nabla) = \frac{i}{2\pi}F_\nabla$ 
            is the curvature two-form of the connection.
        \item $\iota_1$ associates to a $U(1)$ representation of $\pi_1(X)$
            the corresponding flat line bundle.
        \item $\delta_2(\mc{L},\nabla) = c_1(\mc{L})$ is the first Chern character of
            $\mc{L}$.
        \item $\iota_2$ associates to the differential form $\omega$ the
            topologically trivial $U(1)$-bundle with connection $\nabla = d+\omega$.
    \end{itemize} 
    
    Let ${M}_{h,\nabla}$ denote the group (under tensor product) of line
    bundles with an hermitian structure, equipped with an hermitian connection.
    Then, let $\hol:{M}_{h,\nabla}\to \hat{H}^2(X)$ be the isomorphism given by
    taking the holonomy representation.

    \begin{proposition}\label{DCtoLineBundleIsAdditive}
        The isomorphism between the set of hermitian line bundles with
        connection and $\hat{H}^2(X)$ given by the holonomy mapping is a
        group homomorphism.
    \end{proposition} 

    \begin{proof}
        The trivial line bundle with trivial connection $\oh_X$ has zero holonomy, 
        as it is the image of $0$ under $\iota_2$. Hence, $\hol(\oh_X) = 0$.
        Now, let $(\mc{L}_1,\nabla_1)$ and $(\mc{L}_2,\nabla_2)$ be hermitian line bundles
        with connections. Their tensor product is another hermitian line bundle
        with connection, and if $\nabla_1 = d + A_1$ and $\nabla_2 = d + A_2$
        are local trivializations, then
        \begin{equation}\label{tensorProductConnections}
            \nabla_1\otimes\nabla_2 = d + A_1 + A_2
        \end{equation}
        locally. 

        The parallel transport map is thus given locally as
        \begin{equation}
            \begin{aligned}
                P_{\nabla_1\otimes\nabla_2}(\gamma) 
                &= \exp\left(2\pi i \int_\gamma A_1+A_2\right)\\
                &= P_{\nabla_1}(\gamma)P_{\nabla_2}(\gamma)
            \end{aligned} 
        \end{equation} 
        showing that parallel transport along the tensor product is given locally by
        multiplication of the parallel transport maps.
        Since holonomy can be computed locally in this way, we find that
        the holonomy is multiplicative. That is,
        \begin{equation}\label{holonomyMultiplicative}
            \hol_{\nabla_1\otimes\nabla_2} = \hol_{\nabla_1}\hol_{\nabla_2}
        \end{equation}
        and hence, $\hol$ is a group homomorphism.
    \end{proof} 

    %---

\subsection{Complex Tori, Cohomology, and Line Bundles} 

This section follows the excellent exposition of \cite{lange2013complex} and
\cite{polishchuk2003abelian}, except for subsection
\ref{SSSLineBundlesAndDiffCharactersOnTori}, which the author could not
find a suitable reference for.

Much of the main results rely on certain nice properties of line bundles on
complex tori.
Let $V$ be an $n$-dimensional complex vector space and $\Lambda$ a rank $2n$
lattice in $V$. The lattice $\Lambda$ acts on $V$ by addition, and the resulting
quotient space $X = V/\Lambda$ is called a complex torus. Topologically,
$X\cong T^{2n}$ as smooth manifolds, but the complex structure on $V$ descends
to a complex structure on $X$ making it a $\dim_\C(X)=n$-dimensional complex
manifold. The addition rule on $V$ descends to the quotient, and makes $X$ an
complex abelian Lie group. For any $x\in X$, we denote by $t_x:X\to X$ the
translation map by $x$, so that $t_x(y)=x+y$ for all $y\in X$. 

\subsubsection{Cohomology of Complex Tori} %---

The quotient map $\pi:V\to X$ exhibits $V$ as the universal
cover of $X$, and hence $\pi_1(X) (\cong H_1(X,\Z))\cong \Lambda$ canonically.
Furthermore, by the universal coefficient theorem there is a natural isomorphism
$H^1(X,\Z)\cong \Hom(\Lambda,\Z)$. This determines the entire
integral cohomology by the Kunneth formula:
\begin{equation}\label{KunnethTorus}
    H^k(X,\Z) \cong \bigwedge^kH^1(X,\Z)
\end{equation}
where the isomorphism is given right-to-left by taking the cup product. 
The previous two results then imply
    \begin{equation}
        H^k(X,\Z)\cong \bigwedge^k\Hom(\Lambda,\Z)
    \end{equation}.
The universal coefficient formula also computes the complex cohomology groups
as
\begin{equation}\label{eq:kunnethTorusComplex}
        H^k(X,\C)\cong H^k(X,\Z)\otimes\C \cong \Alt^k_\R(V,\C)
    \end{equation}
where $\Alt^k_\R(V,\C)$ is the space of $\R$-multilinear $k$-forms
from $V$ to $\C$.

On Lie groups, there is a certain class of forms which behave well with respect
to the group operation, and will be central in the main construction.
\begin{definition}
    A differential form $\omega\in\mc{A}^k(X)$ is said to be 
    \textit{translation-invariant} if
    \begin{equation}
        t_x^*\omega = \omega
    \end{equation}
    for all $x\in X$. The submodule of all translation-invariant $k$-forms
    is denoted $\IF^k(X)$.
\end{definition} 
\begin{theorem}[\protect{\cite{lange2013complex}*{Proposition~1.3.5}}]\label{CohomologyIsIFTheorem}
    For each cohomology class in $H^k(X,\C)$, there is a unique translation-invariant
    differential $k$-form representing that class. Thus, the de Rham theorem
    furnishes an isomorphism
    \begin{equation}\label{eq:cohomologyIsIF}
        H^k(X,\C)\cong \IF^k(X)
    \end{equation}
    between the $k$-th complex cohomology group and the group $\IF^k(X)$ of
    translation-invariant $k$-forms.
\end{theorem} 

\begin{proof}
    The de Rham isomorphism identifies $H^k(X,\C)$ with the cohomology of
    complex-valued $k$-forms on $X$. If $\lambda_1,\ldots,\lambda_{2n}$
    is a basis for $\Lambda$ and $x_1,\ldots,x_{2n}$ are the corresponding real coordinates
    on $V$, the differentials $dx_i$ are all translation-invariant one-forms on $X$
    and form a basis for $H^1(X,\Z)$ (and hence a basis for $H^1(X,\C)$) dual
    to the chosen basis for $\Lambda\cong H_1(X,\Z)$. By taking wedge products, we
    find that a basis for $H^k(X,\C)$ is given by the $k$-fold wedge products of
    the translation-invariant one-forms $dx_i$. 
    Conversely, each translation-invariant $k$-form must be a real linear combination
    of $k$-fold wedge products of $dx_i$, and we find that each cohomology class
    in $H^k(X,\C)$ has a canonical representative given by a translation-invariant
    $k$-form. That is, there is an isomorphism
    \begin{equation}
        H^k(X,\C)\cong \IF^k(X)
    \end{equation}
    as desired.
\end{proof} 

\begin{remark}
    If the torus $X$ is equipped with the standard flat metric, then the
    translation-invariant forms are exactly the harmonic forms. Thus,
    \cref{CohomologyIsIFTheorem} is a special case of the Hodge theorem. 
\end{remark} 

% --- HODGE THEORY -------------------------------------------------------------
Since $V$--and hence $X$ as well--comes with a complex structure, the cohomology
of $X$ admits a Hodge decomposition. 

Set $\Omega = \Hom_\C(V,\C)$ the space of $\C$-linear maps from $V$ to $\C$
and set $\bar{\Omega} = \Hom_{\bar{\C}}(V,\C)$ the space of $\C$-antilinear
maps, so that $\Hom_\R(V,\C)\cong \Omega\oplus \bar{\Omega}$. This decomposition
induces a decomposition on the higher cohomology groups as
\begin{equation}
    H^k(X,\C)\cong \bigoplus_{p+q=k}\bigwedge^p\Omega\wedge\bigwedge^q\bar{\Omega}
\end{equation}
which coincides with the Hodge decomposition
\begin{equation}
    H^{q}(\Omega_X^p)\cong \bigwedge^p\Omega\wedge\bigwedge^q\bar{\Omega}
\end{equation}.

\subsubsection{Line Bundles on Complex Tori} %---

For this section, all line bundles are assumed to be holomorphic.
Let $\mc{L}$ be a line bundle on $X$. The pullback bundle
$\pi^*\mc{L}$ along the universal covering map is a line bundle on $V$ which is
necessarily the trivial line bundle $V\times \C$. A general fact about covering
spaces is the following:
\begin{proposition}[\protect{\cite{lange2013complex}*{Proposition~B.1}}] 
    Suppose $V$ is the universal cover of $X$ with covering map $\pi$.
    Then, there is a canonical isomorphism
    \begin{equation}
        H^1(\pi_1(X),H^0(\oh_{V}^*))\cong \ker(\pi^*:H^1(X,\oh_X^*)\to H^1(V,\oh_{V}^*))
    \end{equation}
    between the group cohomology of $\pi_1(X)$ in $H^0(\oh_{V}^*)$ and
    the kernel of the pullback map.
\end{proposition}
Elements of the group cohomology $H^1(\Lambda,H^0(\oh_V^*))$ are called
\textit{factors of automorphy} for the covering space, and from the above result
we see that every line bundle on $X$ is described by a factor of automorphy.

Explicitly, a $1$-cocycle from $\Lambda$ to $H^0(\oh_V^*)$ is a holomorphic map
$f:\Lambda\times V\to \C^*$ satisfying the cocycle condition
\begin{equation}
    f(\lambda+\mu,v) = f(\lambda,\mu+v)f(\mu,v)
\end{equation}.
This defines an action of $\Lambda$ on the trivial line bundle $V\times\C$ by
\begin{equation}
    \lambda\cdot(v,z) = (v+\lambda,f(\lambda,v)z)
\end{equation}
for all $\lambda\in\Lambda$. The quotient space $V\times\C/\Lambda$ is then a
line bundle over $X$. 

For what follows, let $f$ be a factor of automorphy describing the line bundle
$\mc{L}$ on $X$. Choose a logarithm
\begin{equation}
    \begin{aligned}
        g&:\Lambda\times V\to \C\\
        f &= \exp(2\pi i g)\\
    \end{aligned}
\end{equation}
for $f$ and consider the alternating two-form $E\in\Alt^2(V,\Z)$ given by
\begin{equation}
    E_L(\lambda,\mu) = g(\mu,\lambda+v) + g(\lambda,v)-g(\lambda,\mu+v)-g(\mu,v)
\end{equation}
for any choice of $v\in V$ (the definition is independent of $v$ as a result
of the cocycle condition on $f$).
\begin{theorem}[\protect{\cite{lange2013complex}*{Theorem~2.1.2}}]
    Under the canonical isomorphism
    \begin{equation}\label{SecondCohomologyTorusIdentification}
        H^2(X,\Z)\cong \bigwedge^2\Hom(\Lambda,\Z)
    \end{equation}
    the first Chern class $c_1(\mc{L})$ maps to the alternating form $E_L$.

    Conversely, every Hermitian form $H:V\times V\to \C$ whose imaginary part
    takes integral values on $\Lambda\times \Lambda$ induces an
    alternating form $E = \im(H)$ which is the first Chern class of a line bundle
    on $X$.
\end{theorem}

Using factors of automorphy all line bundles on $X$ can be characterized.
\begin{definition}
    \label{def:semicharacter}
    A \textit{semi-character} for a Hermitian form $H$ with imaginary part $E$
    taking integral values on $\Lambda$ is a map $\chi:\Lambda\to U(1)$ satisfying
        \begin{equation}
            \chi(\gamma + \mu) = \chi(\gamma)\chi(\mu)\exp(\pi i E(\gamma,\mu))
        \end{equation}.
\end{definition} 
A pair $(H,\chi)$, with $H$ a Hermitian form whose imaginary part takes integral
values on $\Lambda$ and $\chi$ a semi-character for $H$, defines a cocycle
$a_{H,\chi}\in H^1(\Lambda, H^0(\oh_V^*))$ via
    \begin{equation}
        a_{H,\chi}(\gamma, v) = \chi(\gamma)
        \exp(\pi H(v,\gamma) - \frac{\pi}{2}H(\gamma,\gamma))
    \end{equation}.
This in turn defines a line bundle $\mc{L}(H,\chi)$ associated to the pair.
The converse is also true, and this allows for a complete characterization
of $\Pic(X)$.
\begin{theorem}[\protect{Appel-Humbert Theorem, \cite{lange2013complex}*{Theorem~2.2.3}}]\label{appelHumbert}
    Let $P(\Lambda)$ be the group consisting of pairs $(H,\chi)$ of a Hermitian
    form $H$ whose imaginary part takes integral values on $\Lambda$ and a
    semi-character $\chi$, with group operation given by
    \begin{equation}
        (H_1,\chi_1)(H_2,\chi_2) = (H_1+H_2,\chi_1\chi_2)
    \end{equation}.
    Then, the association of $(H,\chi)$ with a line bundle defined above
    induces a group isomorphism between $P(\Lambda)$ and $\Pic(X)$. Furthermore,
    $c_1(\mc{L}(H,\chi)) = H$.
\end{theorem}

In particular, each line bundle $\mc{L}$ on $X$ uniquely determines a pair $(H,\chi)$
in $P(\Lambda)$, and has a canonical choice of representative (the \textit{canonical
factor}) for its class in $H^1(\Lambda,H^0(\oh_V^*))$ given by $a_{H,\chi}$.

\begin{example}
    When $H$ is trivial ($H=0$), the semi-characters are group homomorphisms from
    $\Lambda$ to $U(1)$. Line bundles with $c_1(\mc{L})=0$ are topologically
    trivial, and form rank-one local systems on $X$. The (semi)-character $\chi$
    classifying $\mc{L}$ can then be thought of as a unitary representation
    of $\pi_1(X)\cong\Lambda$ which specifies the monodromy of this system.
\end{example}

This description of line bundles also behaves nicely with respect to the
group operation.
\begin{proposition}[\protect{\cite{lange2013complex}*{Lemma~2.3.2}}]
    If $v\in V$ is any lift of $x\in X$, then
    \begin{equation}
        t_x^*\mc{L}(H,\chi) = \mc{L}(H,\chi \exp(2\pi i \im(H(v,\blank)))
    \end{equation}.
\end{proposition}
This immediately yields the following well-known result:
\begin{theorem}[\protect{Theorem of the square, 
    \cite{lange2013complex}*{Theorem~2.3.3}}]
    \label{thm:square}
    \begin{equation}
        t_{x_1+x_2}^*\mc{L}\cong t_{x_1}^*\mc{L}\otimes t_{x_2}^*\mc{L}\otimes \mc{L}^{-1}
    \end{equation}
\end{theorem}

\subsubsection{Topologically Trivial Line Bundles and the Dual Torus}

Recall that $\Pic^0(X)$ denotes the group of topologically trivial line
bundles on $X$, with group operation the tensor product.
According to the Appel-Humbert theorem (\cref{appelHumbert}), when $H=0$
there exists an isomorphism
\begin{equation}
    \Hom(\Lambda,U(1))\cong \Pic^0(X)
\end{equation}
sending a character $\chi$ to the line bundle $\mc{L}(0,\chi)$. This implies
that topologically $\Pic^0(X)$ is a torus. We can endow it with a natural
complex structure in the following way.

Recall as well that $\overline{\Omega} = \Hom_{\bar{\C}}(V,\C)$ is the vector
space of $\C$-antilinear maps from $V$ to $\C$. It is canonically identified
with the real dual $\hat{V} = \Hom_\R(V,\R)$ by
\begin{equation}
    \begin{aligned}
        \overline{\Omega}&\to \hat{V}\\
        f&\mapsto \Im(f)
    \end{aligned}
\end{equation}.
Thus, there is a canonical nondegenerate pairing
\begin{equation}
    \begin{aligned}
        \la \blank,\blank\ra: \overline{\Omega}\times V &\to \R\\
        \la \varphi, v\ra &= \Im(\phi(v))
    \end{aligned}
\end{equation}
which induces a map
\begin{equation}\label{eq:expCoveringToDualTorus}
    \begin{aligned}
        \exp:\overline{\Omega}&\to \Hom(\Lambda,U(1))\\
            \varphi&\mapsto \exp\left(2\pi i \la \varphi,\blank\ra\right)
    \end{aligned}
\end{equation}
whose kernel $\ker(\exp) = \hat{\Lambda}$ can be identified with the dual
lattice of $\Lambda$ under $\la\blank,\blank\ra$.
\begin{definition}
    The \textit{dual torus} to $X$, denoted $\hat{X}$, is the complex torus
    \begin{equation}
        \hat{X} := \overline{\Omega}/\hat{\Lambda}
    \end{equation}.
\end{definition}

The dual torus can be constructed in various other ways as well.
\begin{proposition}[\protect{\cite{lange2013complex}*{Theorem~2.4.1}}]
    The exponential map \eqref{eq:expCoveringToDualTorus} furnishes an
    isomorphism
    \begin{equation}
        \exp:\hat{X}\to \Hom(\Lambda,U(1))\cong\Pic^0(X)
    \end{equation}.
    Explicitly, a $\C$-antilinear map $\varphi\in\Hom_{\bar{\C}}(V,\C)$
    is mapped to the character $\chi_{\varphi}$ defined by
    \begin{equation}
        \lambda\mapsto \exp\left( 2\pi i \la \varphi,\lambda\ra \right)
    \end{equation}
    which, by the Appel-Humbert theorem, defines the unique line bundle
    $\mc{L}(0,\chi_{\varphi})\in\Pic^0(X)$ with monodromy given by $\chi_{\varphi}$.
\end{proposition}

\begin{corollary}
    There is a natural identification
    \begin{equation}
        H^1(X,U(1))\cong \hat{X}
    \end{equation}
    from the natural isomorphism
    \begin{equation}
        H^1(X,U(1))\cong \Hom(H_1(X,\Z),U(1))=\Hom(\Lambda,U(1))
    \end{equation}
\end{corollary}

\begin{proposition}[\protect{\cite{lange2013complex}*{Theorem~2.5.1}}]
    The dual torus serves as a fine moduli space for topologically trivial
    line bundles on $X$. In particular, there exists a universal bundle $\mc{P}$
    over $X\times\hat{X}$ whose restriction to a fiber 
    $\mc{P}|_{X\times\left[ \hat{x} \right]}$ is isomorphic to the line
    bundle defined by $\hat{x}$.
    Thus, there is a natural identification
    \begin{equation}
        \hat{X} \cong \Pic^0(X)
    \end{equation}
\end{proposition}

\subsubsection{Cohomology of Line Bundles}
\label{SSS:cohomologyOfLineBundles}

The Appel-Humbert theorem gives an explicit description of $\Pic(X)$ as a
collection of $\Pic^0(X)$-torsors indexed by the Chern class.
We can identify the first Chern class of a holomorphic line bundle $\mc{L}$ as 
the imaginary part of an Hermitian form $H$ on $V$. Conversely, every Hermitian
form whose imaginary part takes integral values on $\Lambda$ serves as the
first Chern class for some line bundle. Thus, a line bundle is uniquely specified
by $H$ and a semi-character $\chi$ for $H$.
The linear algebra of $H$ is closely related to the cohomology of $\mc{L}$,
which we examine now, following \cite{lange2013complex}*{Section~2-3}.

\begin{definition}
    For any line bundle $\mc{L} = (H,\chi)$ on $X$, define the map $\phi_{\mc{L}}$
    to be
    \begin{equation}\label{eq:phiLDefn}
        \begin{aligned}
            \phi_{\mc{L}}&:X\to \hat{X}\\
            x&\mapsto t_x^*\mc{L}\otimes\mc{L}^{-1}
        \end{aligned}
    \end{equation}
    which, by the theorem of the square, is a homomorphism. 
\end{definition}

In terms of covering spaces, $\phi$ has an analytic description.
\begin{proposition}[\protect{\cite{lange2013complex}*{Lemma~2.4.5}}]
    The map
    \begin{equation}
        \begin{aligned}
            \phi_H: V&\to\overline{\Omega}\\
            v&\mapsto H(v,\blank)
        \end{aligned} 
    \end{equation}
    is an analytic representation of $\phi_{\mc{L}}$ in the sense
    that $\phi_H$ descends under the quotient to $\phi_{\mc{L}}:X\to\hat{X}$.
\end{proposition} 
Using this analytic description, some basic properties of $\phi_{\mc{L}}$
are derived.
\begin{proposition}[\protect{\cite{lange2013complex}*{Corollary~2.4.6}}]
    The following holds for all $\mc{L}=(H,\chi)$:
    \begin{itemize}
        \item The map $\phi_{\mc{L}}$ only depends on $H$.
        \item The association $\mc{L}\mapsto \phi_{\mc{L}}$ is additive:
            \begin{equation}
                \phi_{\mc{L}_1\otimes\mc{L}_2} = \phi_{\mc{L}_1} + \phi_{\mc{L}_2}
            \end{equation}
    \end{itemize}
\end{proposition}

The kernel of $\phi_{\mc{L}}$ is important and can be computed readily.
Define
\begin{equation}\label{eq:lambdaLDefn}
    \Lambda(\mc{L}) = \{v\in V\ |\ E_H(v,\Lambda)\subset\Z\}
\end{equation}
the set of vectors which pair integrally with $\Lambda$ under $E_H:=\Im(H)$. Then
\begin{equation}
    K(\mc{L}) :=\ker(\phi_{\mc{L}}) = \Lambda(\mc{L})/\Lambda
\end{equation}. By elementary linear algebra, if $E_H$ is nondegenerate
\begin{equation}\label{eq:cardinalityOfKernelPhiL}
    \deg(\phi_{\mc{L}}) = |K(\mc{L})| = \det(E_H)
\end{equation},
and if $\det(E_H)=0$ then the kernel has positive dimension.

By the elementary divisor theorem, there exists a symplectic basis of $\Lambda$
with respect to $E_H$, in the sense that there is a basis
$\lambda_1,\ldots,\lambda_g,\mu_1,\ldots,\mu_g$ for which $E_H$ takes form
\begin{equation}
    \begin{bmatrix}
        0 & D\\
        -D &0
    \end{bmatrix}
\end{equation}
where $D=\text{diag}(d_1,\ldots,d_g)$ is a diagonal matrix with positive
integer entries $d_i$ satisfying $d_i|d_{i+1}$. The integers $d_i$ are uniquely
determined by $E_H$, and are called the \textit{elementary divisors} of $E_H$.
If all $d_i>0$, then $E_H$ is nondegenerate. 
\begin{definition}\label{def:pfaffian}
    The \textit{Pfaffian} of $E_H$, denoted $\Pf(E_H)$, is
    \begin{equation}\label{eq:pfaffianDefn}
        \Pf(E_H) = \det(D)
    \end{equation}
    and if $D$ contains zeroes on the diagonal, the \textit{reduced Pfaffian},
    denoted $\Pfr(E_H)$,
    is the product of the nonzero entries of $D$, unless $D=0$ in which case $\Pfr(D):=1$.
\end{definition}

Through some lengthy computations detailed in \cite{lange2013complex}*{Section~3},
the following cohomological results are arrived at:
\begin{theorem}[\protect{\cite{lange2013complex}*{Corollary~3.2.8}}]
    Suppose $H$ is positive-definite. Then,
    \begin{equation}
        h^0(\mc{L}) = \Pf(E_H)
    \end{equation}
\end{theorem}
\begin{theorem}[\protect{\cite{lange2013complex}*{Theorem~3.3.3}}]
    Suppose $H$ is positive-semidefinite. Denote by $K(\mc{L})_0$ the connected
    component of the identity in $K(\mc{L})$. Then,
    \begin{equation}
        h^0(\mc{L}) = 
        \begin{cases}
            \Pfr(E_H), &\mc{L}|_{K(\mc{L})_0}\text{ is trivial}\\
            0, &\text{else}
        \end{cases}
    \end{equation}
\end{theorem}
\begin{theorem}[\protect{\cite{lange2013complex}*{Theorem~3.5.5}}]
    \label{thm:cohomologyOfLineBundles}
    Let $H$ have $r$ positive and $s$ negative eigenvalues. Then,
    \begin{equation}
        h^q(\mc{L}) =
        \begin{cases}
            {g-r-s \choose q-s}\Pfr(E_H), 
            &s\leq q\leq g-r\text{ and }\mc{L}|_{K(\mc{L})_0}\text{ is trivial}\\
            0, &\text{else}
        \end{cases}
    \end{equation}
\end{theorem}
\begin{example}
    \label{rk:cohomologyofStructure}
    In the case $H=0$, $r=s=0$ and $K(\mc{L})_0 = X$. Then, all cohomology
    groups of $\mc{L}$ vanish unless $\mc{L}$ is trivial on all of $X$, i.e.
    $\mc{L}\cong \oh_X$. Then, 
    \begin{equation}
        h^q(\oh_X)=h^{0,q}(X)={g \choose q} 
    \end{equation} 
    as expected.
\end{example} 

\begin{example}
    \label{ex:cohomologyOfAmple}
    In the case of $r+s=g$, the only nontrivial cohomology group is
    $q=s$, which has dimension $\Pf(E_H)$. In other words,
    if $\mc{L}$ is a line bundle whose corresponding Hermitian form is
    nondegenerate, then
    \begin{equation}
        H^q(X,\mc{L}) =
        \begin{cases}
            \Pf(E_H) &q=s\\
            0 &\text{else}
        \end{cases} 
    \end{equation}.
\end{example} 

\subsubsection{Line Bundles and Differential Characters on Tori} %---
\label{SSSLineBundlesAndDiffCharactersOnTori}

This subsection is believed to be well-known, but the author could not find
a citation for it. 

Let $\mc{L} = (H,\chi)$ be a holomorphic line bundle with Appel-Humbert data
given by $(H,\chi)$, and let $\nabla$ be a connection on
$\mc{L}$. Denote by $E_H = \im(H)\in H^2(X,\Z)$ the corresponding Chern class.
\begin{definition}
    Fix a cohomology class $\omega\in H^2(X,\Z)$. A line bundle with connection
    $(\mc{L},\nabla)$ is said to be \textit{$\omega$-flat} if $c_1(\mc{L})= \omega$ and
    \begin{equation}\label{EFlatDefn}
        F_\nabla := \nabla^2 \in IF^2(X)
    \end{equation}
    where $\IF^2(X)$ is the space of translation-invariant forms on $X$.
\end{definition} 
\begin{remark}
    In this notation, an $E_H$-flat connection is an invariant connection
    on a line bundle $\mc{L}$ with associated hermitian form $H$ and Chern
    class $c_1(\mc{L})=E_H$.
\end{remark}

\begin{proposition}
    \label{prop:flatConnectionsAreTorsor}
    The space of $E_H$-flat connections on the complex bundle $\mc{L}$ forms in
    a natural way a torsor over $\Pic^0(X)$. In particular,
    there is a unique $E_H$-flat connection on $\mc{L}$ which is compatible
    with the holomorphic structure.
\end{proposition} 
\begin{proof}

    Recall that $H:V\times V\to \C$ is an hermitian form on $V$ whose imaginary
    part $E_H :=\im(H)$ is the first Chern class of $\mc{L}$ under the identification
    \begin{equation}
        H^2(X,\Z)\cong\bigwedge^2\Hom(\Lambda,\Z)
    \end{equation}
    from the Kunneth formula \eqref{KunnethTorus}. 
    By \cref{CohomologyIsIFTheorem}, there is a unique translation-invariant
    form $F_H\in IF^2(X)$ whose cohomology class coincides with $E_H$.
    A connection with this curvature can be written locally as
    \begin{equation}
        \begin{aligned}
            \nabla &= d + A\\
            dA &= F_H
        \end{aligned} 
    \end{equation}
    which can be solved for $A$. Thus, $E_H$-flat connections exist.

    For $E\in H^2(X,\Z)$, associate to it the unique translation-invariant form
    $E\in IF^2(X)$ whose cohomology is $E$. Then, the space of $E$-flat connections
    is precisely the preimage of $E$ under the curvature map.
    The group $H^1(X,U(1))$
    acts on this preimage in the following way: for each character $\chi\in
    H^1(X,U(1))$, associate to it the unique line bundle
    $(\mc{L}_\chi,\nabla_\chi)$ with flat connection whose monodromy is given
    by $\chi$. Then, the action of $\chi$ on $(\mc{L},\nabla)$ is given by
    tensor product with $\mc{L}_\chi$.
    This is, indeed, well-defined:
    \begin{equation}
        \begin{aligned}
            c_1(\mc{L}_\chi\otimes\mc{L}) &= c_1(\mc{L}_\chi)+c_1(\mc{L})
            =c_1(\mc{L}) = E\\
            F_{\nabla_\chi\otimes\nabla} &= F_{\nabla_\chi} + F_{\nabla}
            =F_\nabla = E
        \end{aligned} 
    \end{equation}.

    The action is also free and transitive: if $(\mc{L}_1,\nabla_1)$ and
    $(\mc{L}_2,\nabla_2)$ are two line bundles with $E$-flat connections, then
    $\mc{L}_3 = \mc{L}_1\otimes\mc{L}_2^{-1}$ is a trivial line bundle with flat connection
    $\nabla_3 = \nabla_1\otimes (-\nabla_2)$,
    and 
    \begin{equation}
        \begin{aligned}
            \mc{L}_3\otimes\mc{L}_2 &= \mc{L}_1\otimes\mc{L}_2^{-1}\otimes\mc{L}_2\\
                                    &=\mc{L}_1\\
            \nabla_3\otimes\nabla_2 &= d + A_1 +A_2 - A_2 &\text{locally}\\
                                    &= \nabla_1
        \end{aligned} 
    \end{equation}
    showing that the action of $H^1(X,U(1))$ is transitive. A similar argument
    shows the action is also free.

    Hence, the space of $E$-flat connections is a torsor over $H^1(X,U(1))$ as
    desired.
\end{proof} 

When $E=0$, an $E$-flat connection is simply a flat connection. The
Riemann-Hilbert correspondence then allows us to identify the space of
flat connections with the monodromy representations of $\pi_1(X)$. 
This identification extends to $E$-flat connections for arbitrary integral
classes.
\begin{theorem}
    \label{thm:EFlatIsHolomorphicIsomorphism}
    The space of $E$-flat connections is canonically isomorphic to
    $\Pic^{E}(X) := c_1^{-1}(E)$. 
\end{theorem} 
\begin{proof}
    By \cite{polishchuk2003abelian}*{Proposition~1.5},
    every holomorphic line bundle $\mc{L}=(H,\chi)$ on $X$ comes equipped with
    a natural hermitian metric defined on the covering space $V$ as
    \begin{equation}\label{eFlatProof:HermitianForm}
        h(v) = \exp(-\pi H(v,v))
    \end{equation}
    and the Chern connection of this hermitian structure is $E_H$-flat.

    Let us follow this argument more closely. Recall that $V\to X$ is the
    universal cover of $X$, and $\mc{L}$ is the quotient of the trivial bundle
    on $V$ by the action of $\Lambda=H_1(X,\Z)$ defined by
    \begin{equation}\label{eFlatProof:appelHumbertQuotientEqn}
        \lambda\cdot (v,z) = (v+\lambda, \chi(\lambda)
        \exp\left( \pi H(v,\lambda) - \frac{\pi}{2}H(\lambda,\lambda) \right) z)
    \end{equation}.
    By \eqref{chernConnectionEquation} the connection is explicitly
    \begin{equation}\label{eFlatProof:connection}
        \nabla = d -\pi H(dv,v)
    \end{equation}
    where $H(dv,v)$ denotes (following the notation of \cite{polishchuk2003abelian})
    the one-form defined as follows: choose linear coordinates
    on $V$ so that $H$ is diagonalized: $H(v,w) = \sum_i v_i\bar{w}_i$. Then
    for $v = (z_1,\ldots,z_n)$, $H(dv,v) = \sum_i \bar{z}_idz_i$.

    The curvature of this connection is then readily computed:
    \begin{equation}\label{eFlatProof:curvature}
        F_\nabla = \pi H(dv,dv)
    \end{equation}
    and, noting that $H(dv,dv)$ is purely imaginary,
    this is equal to $\pi i E_H(dv,dv) = \pi i \sum_i dz_i\wedge d\bar{z}_i$ as
    an invariant two-form.

    Thus, the association of a holomorphic line bundle $\mc{L} = (H,\chi)$
    to the Chern connection associated to the canonical hermitian form
    \eqref{eFlatProof:HermitianForm} yields a map
    \begin{equation}
        \varphi:\Pic^E(X)\to \curv^{-1}(E_H)\subseteq \hat{H}(X)
    \end{equation}
    whose image lies in the space of $E_H$-flat connections, as
    \eqref{eFlatProof:curvature} shows.

    Under the natural action of $\Pic^0(X)$ on the space of $E_H$-flat connections,
    defined in \cref{prop:flatConnectionsAreTorsor}, this association is
    equivariant. Indeed, for $\mc{L}_0=(0,\chi_0)\in\Pic^0(X)$ and
    $\mc{L} = (H,\chi)\in \Pic^E(X)$,
    \begin{equation}
        \mc{L}_0\cdot\mc{L} = \mc{L}_0\otimes\mc{L} = (H,\chi_0\chi)
    \end{equation}
    and this action does not change the hermitian form
    \eqref{eFlatProof:HermitianForm} on the covering space. 
    Thus, $\varphi(\mc{L}_0\cdot\mc{L})$ is computed by multiplying the holonomy
    by $\chi_0$, which implies
    \begin{equation}
        \varphi(\mc{L}_0\cdot\mc{L}) = \mc{L}_0\cdot\varphi(\mc{L})
    \end{equation}
    as desired.
    Since a $G$-equivariant map between $G$-torsors is an isomorphism, the
    assertion is proved.
\end{proof} 

Differential characters on tori admit a particularly nice description in
terms of semi-characters, using the top-left to bottom-right exact sequence
of \eqref{diffCharacterDiagram}.
\begin{theorem}
    \label{thm:structureDConTori}
    Fix a basis $\gamma_i$ for $\Lambda$, and regard elements of $\Lambda$ as
    elements of $Z_1(X)$ by associating to any primitive $\lambda\in\Lambda$
    the linear subtorus along the vector $\lambda$. Fix as well a closed two-form
    $\omega$ on $X$ with integral periods.
    For any choice $z_i\in U(1)$ of values in $U(1)$ for each $\gamma_i$,
    there is a unique differential character $f\in \hat{H}^2(X)$ with
    \begin{equation}
        \curv(f) = \omega
    \end{equation}
    and 
    \begin{equation}
        f(\gamma_i) = z_i
    \end{equation}
    for each $i$.
\end{theorem} 
\begin{proof}
    Pick any $g\in\hat{H}^2(X)$ with $\curv(g)=\omega$. Such a $g$ exists
    by surjectivity of $\delta_1$ in \eqref{diffCharacterDiagram}.
    Then, the regular action of $H^1(X,U(1))\cong\Pic^0(X)$ on $\hat{H}^2(X)$
    leaves the fibers of $\delta_1$ invariant. In particular, let $\chi\in H^1(X,U(1))$
    be the character taking values on $\gamma_i$ as
    \begin{equation}
        \chi(\gamma_i) = \frac{z_i}{g(\gamma_i)}
    \end{equation}.
    Then, the differential character $\chi+g$ satisfies
    \begin{equation}
        \curv(\chi+g) = \curv(g) = \omega
    \end{equation}
    and
    \begin{equation}
        (\chi+g)(\gamma_i) = \chi(\gamma_i)g(\gamma_i) = z_i
    \end{equation}
    as desired.
    
    Clearly such a differential character is unique. If $f_1,f_2$ are such
    that
    \begin{equation}
        \curv(f_1) = \curv(f_2)
    \end{equation}
    and
    \begin{equation}
        f_1(\gamma_i)=f_2(\gamma_i)
    \end{equation}
    then their difference satisfies
    \begin{equation}
        \curv(f_1-f_2) = 0
    \end{equation}
    and
    \begin{equation}
        (f_1-f_2)(\gamma_i) = 1
    \end{equation}.
    But since $f_1-f_2$ is flat, it is determined by the value on basis elements.
    Thus, $f_1-f_2 = 0$ and $f_1=f_2$ as desired.
\end{proof}

\subsection{Generalized Complex Geometry} %---
    
In \cite{gualtieri2004generalized}, Gualtieri builds the theory of \textit{
generalized complex geometry}, which unifies complex geometry and
symplectic geometry in the broad context of A and B type supersymmetry
We follow Gualtieri's exposition for subsections \ref{SSS:basicConstructions}
and \ref{SSS:generalizedMetrics}. 

\subsubsection{Basic Constructions}
\label{SSS:basicConstructions}

Let $X$ be a smooth manifold of dimension $n$. Generalized complex geometry
takes place on the bundle $TX\oplus T^*X$ over $X$. On this bundle, there is a
canonical neutral (signature $(n,n)$) metric defined by evaluation:
\begin{equation}
    \la X+\xi,Y+\eta\ra = \frac{1}{2}(\xi(Y) + \eta(X))
\end{equation}
and the musical isomorphism from this metric instantiates the canonical isomorphism
\begin{equation}
    TX\oplus T^*X \cong T^*X\oplus TX
\end{equation}.

\begin{definition}[\protect{\cite{gualtieri2004generalized}*{Definition~4.14}}]
    \label{def:almostGCStructure}
    An \textit{almost generalized complex structure} on $X$ is an endomorphism
    $\mc{J}\in\End(TX\oplus T^*X)$ satisfying
    \begin{equation}
        \label{eq:almostGCSquaresTo-1}
        \mc{J}^2 = -1
    \end{equation}
    and
    \begin{equation}\label{eq:almostGCIsOrthogonal}
        \mc{J}^* = -\mc{J}
    \end{equation}
    where $\mc{J}^*$ is the adjoint of $\mc{J}$ with respect to the canonical metric.
\end{definition}
The second condition is equivalent to $\mc{J}$ being orthogonal with respect to
the canonical metric, since
\begin{equation}
    \mc{J}^*\mc{J} = (-\mc{J})\mc{J} = -(-1) = 1
\end{equation}.

\begin{example}
    \label{ex:complexStructureGC}
    Recall that a complex structure on $TX$ is an endomorphism $J$ of $TX$ satisfying
    $J^2=-1$.
    Any complex structure on $X$ is an almost generalized complex structure: take
    \begin{equation}
        \mc{J}_J = 
        \begin{bmatrix}
            -J&0\\
            0&J^*
        \end{bmatrix}
    \end{equation}
\end{example}
\begin{example}
    \label{ex:symplecticStructureGC}
    Any symplectic structure on $X$ is an almost generalized complex structure: if
    $\omega$ is the symplectic form, then take
    \begin{equation}
        \mc{J}_\omega = 
        \begin{bmatrix}
            0&-\omega^{-1}\\
            \omega&0
        \end{bmatrix}
    \end{equation}
\end{example}

The adjective \textit{almost} suggests that there is an additional integrability
condition that should be imposed. In the case of \cref{ex:complexStructureGC}
integrability of a complex structure corresponds to Frobenius integrability
of the $+i$-eigenspace $T^{1,0}X$ of $X$, i.e.\ 
$\left[ T^{1,0}X,T^{1,0}X\right] \subseteq T^{1,0}X$. In the case of
\cref{ex:symplecticStructureGC} the ``integrability'' condition is $d\omega = 0$.
Both of these are special cases of a general notion of integrability with respect
to a special bracket operation on $TX\oplus T^*X$.
\begin{definition}[\protect{\cite{gualtieri2004generalized}*{Section~3.2}}]
    The \textit{Courant bracket} on $TX\oplus T^*X$ is the skew-symmetric bracket
    defined by
    \begin{equation}
        \left[ X+\xi, Y+\eta \right]_C = \left[ X,Y \right]+\mc{L}_X\eta-\mc{L}_Y\xi
        -\frac{1}{2}d\left( \iota_X\eta-\iota_Y\xi \right)
    \end{equation}
    where $\mc{L}_X$ is the Lie derivative along $X$, and $\iota_X$ is contraction
    with $X$.
\end{definition}

\begin{definition}[\protect{\cite{gualtieri2004generalized}*{Definition~4.18}}]
    An almost generalized complex structure $\mc{J}$ is said to be
    \textit{integrable}, or a \textit{generalized complex structure} if the
    $+i$-eigenbundle of $\mc{J}$ is involutive under the Courant bracket.
\end{definition}

\begin{proposition}[\protect{\cite{gualtieri2004generalized}*{Example~4.20-4.21}}]
    An almost generalized complex structure induced by an almost complex structure,
    as in \cref{ex:complexStructureGC}, is integrable if and only if the
    complex structure is integrable.

    Similarly, an almost generalized complex structure induced by a nondegenerate
    two form, as in \cref{ex:symplecticStructureGC}, is integrable if and only
    if the two-form is closed.
\end{proposition}

\subsubsection{Generalized Metrics}
\label{SSS:generalizedMetrics}
\begin{definition}[\protect{\cite{gualtieri2004generalized}*{Section~6.1}}]
    \label{def:generalizedMetric}
    A \textit{generalized metric} on $X$ is a positive-definite metric $G$ on
    $TX\oplus T^*X$ compatible with the natural inner product. That is,
    $G^2=1$.
\end{definition}
 Equivalently, a
generalized metric is the data of a middle-dimensional subbundle $C_+\subseteq
TX\oplus T^*X$ which is positive-definite with respect to the natural inner
product. Define $C_-$ to be the orthogonal complement of $C^+$, and we can
define $G$ to be
    \begin{equation}
        G=\la ,\ra|_{C_+} - \la ,\ra|_{C_-}
    \end{equation}.

A generalized metric is said to be compatible with a generalized complex
structure $\mc{J}$ if the eigenspaces $C_{\pm}$ are preserved by $\mc{J}$. That
is, if $\mc{J}$ is orthogonal with respect to $G$. In this case, since $G$ and
$\mc{J}$ commute, the product $G\mc{J} = \mc{J}G$ defines a second almost generalized
complex structure denoted $\genK$.
\begin{definition}
    \label{def:generalizedKahler}
    A pair $(G,\mc{J})$ of a generalized metric and an almost generalized
    complex structure $\mc{J}$ compatible with $G$ is said to be a
    \textit{generalized K\"ahler structure} on $X$ if both $\mc{J}$ and $\genK$
    are integrable.

    In such a case, the endomorphism $\genK$ is called the \textit{generalized
    K\"ahler form}.
\end{definition}

These definitions are motivated by the following illuminating example.
\begin{example}[\protect{\cite{gualtieri2004generalized}*{Example~6.4}}]
    Let $(X,g,J,\omega)$ be a K\"ahler manifold. Then, $J$ and $\omega$
    define generalized complex structures $\mc{J}_J$ and $\mc{J}_{\omega}$.
    Their product defines a positive-definite metric $G$
    which, due to the K\"ahler condition, can be written in block matrix form
    as
    \begin{equation}
        G = 
        \begin{bmatrix}
            0 &g^{-1}\\
            g &0
        \end{bmatrix}
    \end{equation}
    which is the metric induced by $g$.
    The pair $(G,\mc{J}_J)$ then defines a generalized K\"ahler structure.
\end{example}

\subsubsection{Submanifolds and Generalized Tangent Spaces}
\label{SSS:submanifoldsAndGCTangent}

Let $M$ be a submanifold of $X$. Just as $TM$ is a subbundle of $TX$,
there is a natural subbundle of $TX\oplus T^*X$ associated to $M$.
The first choice for the generalized tangent space of $M$
would be the subbundle of $(TX\oplus T^*X)|_M$:
\begin{equation}
    \mc{T}_0M = TM\oplus \Ann(TM)
\end{equation}
which is a maximal isotropic subbundle of $TX\oplus T^*X$.
However, this definition is not stable under the symmetries of
$TX\oplus T^*X$ (as discussed in \cite{gualtieri2004generalized}*{Section~7.1}).
The correct definition is
\begin{definition}
    Let $M$ be a submanifold of $X$ and $F$ a closed two-form on $M$. Then,
    the \textit{generalized tangent bundle} of $(M,F)$, denoted $\mc{T}^FM$
    (or $\mc{T}M$ if the context is clear),
    is the subbundle of $(TX\oplus T^*X)|_M$ given by
    \begin{equation}
        \mc{T}^FM = \{ X+\xi\ |\ X\in TM, \xi|_M = \iota_X F\}
    \end{equation}.
\end{definition}

As in the case of complex structures, the generalized tangent bundle identifies
if a submanifold is a generalized complex submanifold:
\begin{definition}
    Let $M$ be a submanifold of $X$ along with a closed two-form $F$ on $M$,
    and let $\mc{J}$ be a generalized complex structure on $X$.
    The pair $(M,F)$ is said to be a \textit{generalized complex submanifold}
    if $\mc{T}^FM$ is stable under $\mc{J}$.
\end{definition}

\begin{example}
    \label{ex:GCsubmanifoldComplex}
    Suppose $J$ is a complex structure on $X$ and $\mc{J}$ is the induced
    generalized complex structure, as in \cref{ex:complexStructureGC}.
    Then, a pair $(M,F)$ is a generalized complex submanifold
    if and only if $M$ is a holomorphic submanifold of $X$ and
    $F$ is purely of type $(1,1)$.

    If $F$ has integral periods, then it is the curvature of a unitary connection
    on $M$ which induces an hermitian holomorphic line bundle for which 
    $F$ is the curvature of the Chern connection.
\end{example}

\begin{example}
    \label{ex:GCsubmanifoldSymplectic}
    Suppose $\omega$ is a symplectic form on $X$ and $\mc{J}$ is the induced
    generalized complex structure, as in \cref{ex:symplecticStructureGC}.
    Then, a pair $(M,F)$ is a generalized complex submanifold
    if and only if either
    \begin{itemize}
        \item $M$ is Lagrangian and $F=0$, or
        \item $M$ is coisotropic and $F$ satisfies the conditions of a rank-one
            coisotropic brane in the sense of \cite{Kapustin:2001ij}.
    \end{itemize}
\end{example}

%=== Chapter Main Construction ---

\section{Main Construction}
\label{chapter:main}

The ideas of double field theory suggest that it is possible to examine
D-branes on a torus by lifting to a double-dimensional space on which the
symmetry group is enlarged to include $T$-duality transformations. 
This program was carried out in the topological $A$-model in
\cite{qin2020coisotropic}, where it was discovered that the Floer cohomology,
appropriately modified, computed the $A$-model open string spectra.

We propose a similar construction for the topological $B$-model, using Y.\
Qin's construction of the doubling space. The lift of a space-filling rank one
D-brane in the $B$-model is defined, and the link between
the generalized complex geometry of the base and the
geometry of the lift is established. We show that, in simple cases,
intersection theory computes the correct open string spectra.

\subsection{The Doubling Torus}

Let $V$ be an $n$-dimensional complex vector space, $\Lambda$ a
lattice in $V$, and $X=V/\Lambda$ the corresponding complex torus. We denote by
$\hat{X} := \Pic^0(X)$ the dual torus. 
The complex structure on $X$ will be denoted by $J\in\End(TX), J^2=-1$, and
the dual complex structure on $\hat{X}$ will be denoted by $-J^*$. 
The product $X\times\hat{X}$ supports
the \textit{Poincar\'e bundle} $\mc{P}$ which expresses $\hat{X}$ as the moduli
space of topologically trivial line bundles on $X$.

For the topological $B$-model on $X$, we propose the following definition
of the doubling space of $X$:
\begin{definition}
    \label{def:doublingTorus}
    The \textit{doubling torus}%
    \footnote{Also referred to by some as the \textit{doubling space} or
    \textit{doubled torus}.}
    of $X$, denoted $\mathbb{X}$, is the 
    complex manifold
    \begin{equation}\label{eq:doublingTorus}
        \mathbb{X} = X\times\hat{X}
    \end{equation}
    with the product complex structure,
    equipped with a closed nondegenerate two-form $c_1(\mc{P})$, the first
    Chern class of the Poincar\'e bundle. 
\end{definition}

\begin{remark}
    In \cite{qin2020coisotropic}*{Definition~3.1}, Y.\ Qin considers a symplectic
    torus $(X,\omega)$ instead of a complex torus, and equips the doubling
    space with the natural symplectic form $\omega\oplus -\omega^{-1}$. This
    reflects the fact that the topological $A$-model is defined on symplectic
    manifolds whereas the topological $B$-model is defined on complex manifolds.
\end{remark}

The two-form on the doubling torus should be thought of as measuring
the product structure of $X\times\hat{X}$. Indeed, if $e_i$ are a basis for
$H^1(X,\Z)$ with dual basis $e_i^*$ for $H^1(\hat{X},\Z)$, then
\begin{equation}\label{eq:poincareBundleInCoords}
    c_1(\mc{P}) = \sum_i e_i\wedge e_i^*
\end{equation}. This is alternatively thought of as the
fundamental two-form defining a para-Hermitian structure on $\mathbb{X}$.

\subsection{Lifting Generalized Complex Structures}
    
    As noted in \cref{def:doublingTorus}, the doubling torus comes with a natural
    complex structure inherited from the complex structure on $X$ and the
    dual complex structure on $\hat{X}$. Furthermore, a symplectic structure
    on $X$ induces an inverse symplectic structure on $\hat{X}$, and in this
    case the doubling torus inherits a symplectic structure as well.
    As Y.\ Qin observes (\cite{qin2020coisotropic}*{Remark~3.2}), this
    doubled symplectic form functions as a different complex structure on the
    doubling space. This is reminiscent of the constructions of generalized
    complex geometry, and we make that connection explicit now.

    \subsubsection{The Tangent Bundle of the doubling torus}
    
    Let $\mathbb{X}$ be the doubling torus of $X$, equipped with projections
    \begin{equation}
        \begin{tikzcd}
            &\mathbb{X}\arrow[ld, "\pi"']\arrow[rd, "\hat{\pi}"]\\
            X &&\hat{X}
        \end{tikzcd}
    \end{equation}
    and symplectic form
    \begin{equation}
        \sigma = c_1(\mc{P})
    \end{equation}.
    From the local description of $\sigma$ in \eqref{eq:poincareBundleInCoords},
    it is clear that the fibers of $\pi$ and $\hat{\pi}$ are transverse Lagrangian
    foliations of $\hat{X}$. 

    We now make clear the connection to generalized complex geometry with the
    following theorem.
    \begin{theorem}
        \label{thm:doubledTangentIsGCTangent}
        The tangent bundle of $\mathbb{X}$ is canonically isomorphic
        (under $\sigma$) to the pullback of the sum of tangent and cotangent bundles
        of either $X$ or $\hat{X}$. That is,
        \begin{equation}\label{eq:doubledTangentIsGCTangent}
            T\mathbb{X}\cong \pi^*(TX\oplus T^*X)\cong \hat{\pi}^*(T\hat{X}\oplus T^*\hat{X})
        \end{equation}
        as smooth vector bundles.
    \end{theorem} 

    \begin{proof}
        Since $\mathbb{X}=X\times\hat{X}$ is globally a product
        manifold, its tangent bundle splits
        \begin{equation}\label{eq:tangentBundleDoubledSplits}
            T\mathbb{X}\cong \pi^*TX\oplus \hat{\pi}^*T\hat{X}
        \end{equation}
        which are, from a different perspective, the tangent bundles of the two
        Lagrangian foliations of $\mathbb{X}$.
        Since both foliations are Lagrangian, contraction with the symplectic
        form $\sigma$ induces an isomorphism between the normal bundles and the
        cotangent bundles. 
        Applying this to either foliation results in the isomorphisms
        \begin{equation}\label{eq:doubledTangentSplitsTangentCotangent}
            \begin{aligned}
                \pi^*TX\oplus\hat{\pi}^*T\hat{X}
                &\cong \pi^*TX\oplus\pi^*T^*X\\
                (v,\hat{v})&\mapsto (v,\iota_{\hat{v}}\sigma)
            \end{aligned} 
        \end{equation}
        and
        \begin{equation}
            \begin{aligned}
                \pi^*TX\oplus\hat{\pi}^*T\hat{X}
                &\cong \hat{\pi}^*T^*\hat{X}\oplus \hat{\pi}^*T\hat{X}\\
                (v,\hat{v})&\mapsto (\iota_v\sigma, \hat{v})
            \end{aligned} 
        \end{equation}.
        In both cases, we have implicitly used the isomorphism
        \begin{equation}
            (\pi^* E)^* \cong \pi^* (E^*)
        \end{equation}
        between the dual of the pullback and the pullback of the dual.
    \end{proof}

    \subsubsection{Lifting Generalized Complex Structures}
    \label{SSS:liftingGC}
    The identification in \cref{thm:doubledTangentIsGCTangent}
    allows for generalized complex geometry on $X$ to be transported to 
    ordinary geometry on $\mathbb{X}$. Many of the natural constructions
    of generalized geometry extend in this way, which we examine now.

    \begin{definition}
        The \textit{foliation operator}, denoted by $F$, is the endomorphism
        of $T\mathbb{X}\cong \pi^*TX\oplus\hat{\pi}^*\hat{X}$ defined
        as taking the value $+1$ on $\pi^*TX$ and $-1$ on $\hat{\pi}^*T\hat{X}$.
    \end{definition}
    The foliation operator is alternatively defined by the pair of projectors
    $\Pi,\hat{\Pi}$ onto the two factors of $T\mathbb{X}$ by
    \begin{equation}
        F = \Pi-\hat{\Pi}
    \end{equation}
    and the association goes in reverse. Any operator $F$ with $F^2=1$
    defines projectors by
    \begin{equation}
        \begin{aligned}
            \Pi         &= \frac{1}{2}\left( 1 + F \right)\\
            \hat{\Pi}   &= \frac{1}{2}\left( 1 - F \right)
        \end{aligned}
    \end{equation}.
    These define transverse foliations on $\mathbb{X}$ if the corresponding
    distributions are Frobenius integrable. 

    Observe that for any endomorphism
    $\oh\in \End(T\mathbb{X})$ with block-matrix representation
    \begin{equation}
        \oh = 
        \begin{bmatrix}
            A &B\\
            C &D
        \end{bmatrix} 
    \end{equation}
    composition with $F$ yields
    \begin{equation}
        F\oh = 
        \begin{bmatrix}
            A  &B\\
            -C &-D
        \end{bmatrix} 
    \end{equation}.
    In particular, the symplectic form $\sigma$ takes block-diagonal form
    \begin{equation}
        \sigma =
        \begin{bmatrix}
            0 & \sigma|_X\\
            -\sigma|_X &0
        \end{bmatrix} 
    \end{equation}
    and post-composition with $F$ yields the symmetric form
    \begin{equation}
        F\sigma =
        \begin{bmatrix}
            0 &\sigma|_X\\
            \sigma|_X &0
        \end{bmatrix} 
    \end{equation}. 
    Recall as well that the canonical neutral metric on $TX\oplus T^*X$ is given by
    the evaluation map
    \begin{equation}
        \begin{aligned}
            \la \blank,\blank\ra: TX\oplus T^*X&\to\R\\
            \la X,\omega\ra &= \omega(X)
        \end{aligned} 
    \end{equation}.
    \begin{proposition}
        The canonical neutral metric on $TX\oplus T^*X$ lifts to the metric defined
        by $F\sigma$. 
    \end{proposition} 
    \begin{proof}
        By direct computation,
        \begin{equation}
            \begin{aligned}
                F\sigma(v+\hat{v})(w+\hat{w})
                &= \sigma(-v+\hat{v},w+\hat{w})\\
                &=\sigma(\hat{v},w) + \sigma(\hat{w},v)\\
                &=((\iota_{\hat{v}}\sigma)(w) + (\iota_{\hat{w}}\sigma)(v))
            \end{aligned} 
        \end{equation}
        which is, by virtue of the isomorphism \eqref{eq:doubledTangentSplitsTangentCotangent}
        exactly the pairing given by the evaluation map.
    \end{proof} 
    \begin{remark}
        Coordinatize $\mathbb{X}$ via the coordinates defined in
        \eqref{eq:poincareBundleInCoords}. In these coordinates, the
        canonical neutral metric is thus
        \begin{equation}\label{eq:canonicalMetricInCoords}
            F\sigma = \sum_i e_i\otimes e_i^*
        \end{equation}.
    \end{remark}
    \begin{remark}
        The musical isomorphisms between $T\mathbb{X}$ and $T^*\mathbb{X}$
        given by the neutral metric express the natural isomorphism
        \begin{equation}
            \pi^*\left( TX\oplus T^*X \right)
            \cong \pi^*\left( T^*X\oplus TX \right)
        \end{equation}.
    \end{remark}

    \begin{proposition}
        \label{prop:liftingGCStructures}
        Let $\mc{J}\in\End(TX\oplus T^*X)$ be an almost generalized complex structure
        on $X$. Then, $\mc{J}$ lifts to an almost complex structure $\pi^*\mc{J}$ on
        $\mathbb{X}$.
    \end{proposition} 
    \begin{proof}
        This is a direct consequence of \cref{thm:doubledTangentIsGCTangent}.
        Namely, let $\mc{J}:TX\oplus T^*X\to TX\oplus T^*X$ be such that $\mc{J}^2=-1$
        and $\mc{J}^*=-\mc{J}$ under the canonical pairing on $TX\oplus T^*X$.
        This induces an endomorphism of $T\mathbb{X}\cong \pi^*(TX\oplus T^*X)$
        by pullback:
        \begin{equation}
            \pi^*(\mc{J}):T\mathbb{X}\to T\mathbb{X}
        \end{equation}
        and by functoriality of $\pi^*$ this satisfies $\pi^*(\mc{J})^2=-1$
        as desired.
    \end{proof} 

    \begin{example}[Lifts of Symplectic Structures, c.f.\ \protect{\cite{qin2020coisotropic}*{Remark~3.2}}]
        Let $\omega$ be a symplectic form on $X$. Then $\omega$ induces
        a generalized complex structure (\cref{ex:symplecticStructureGC})
        defined in block matrix form as
        \begin{equation}
            \mc{J}_\omega = 
            \begin{bmatrix}
                0 &\omega^{-1}\\
                -\omega &0
            \end{bmatrix} 
        \end{equation}.
        This lifts to a complex structure on $\mathbb{X}$ 
        which is a central object of study in \cite{qin2020coisotropic}.
    \end{example} 

    \begin{example}
        \label{ex:liftComplexStructure}
        Let $J$ be a complex structure on $X$. Then $J$ induces
        a generalized complex structure (\cref{ex:complexStructureGC})
        defined in block matrix form as
        \begin{equation}
            \mc{J}_J = 
            \begin{bmatrix}
                J &0\\
                0  &-J^*
            \end{bmatrix} 
        \end{equation}.
        From the construction of $\hat{X}$, we see that $-J^*$ is the
        complex structure on $\hat{X}$ and thus $\mc{J}_J$ is simply the
        induced complex structure on the product.
    \end{example} 

    \begin{example}
        [Lifts of Metrics]
        Suppose we equip $X$ with a Riemannian metric $g$. Then,
        $g$ induces a metric $g^{-1}$ on $\hat{X}$ and the pair forms into
        a Riemannian metric on $\mathbb{X}$. Using 
        the canonical metric, this induces an endomorphism $G\in\End(T\mathbb{X})$
        defined in block matrix form as
        \begin{equation}
            G = 
            \begin{bmatrix}
                0& g^{-1}\\
                g &0
            \end{bmatrix}
        \end{equation}
        which satisfies the algebraic properties of a generalized metric.
        
        If $X$ is equipped with a B-field $B\in H^2(X,\R)$,
        then the B-shifted generalized metric is given by
        \begin{equation}
            G^B = 
            \begin{bmatrix}
                -g^{-1}B & g^{-1}\\
                g-Bg^{-1}B &Bg^{-1}
            \end{bmatrix}
        \end{equation}
        which lifts as well to a positive-definite metric on $\mathbb{X}$.
        
    \end{example}

    \begin{example}[Lifts of K\"ahler Structures]
        \label{ex:liftsKahler}
        If $(g,J,\omega)$ is a K\"ahler structure on $X$, the lifts of $G$,
        $\mc{J}_J$, and $\mc{J}_\omega$ on $\mathbb{X}$ define a K\"ahler
        structure on $\mathbb{X}$. Indeed, the lift $G$ of $g$ is a
        positive-definite metric and by direct computation
        $\mc{J}_\omega=G\mc{J}_J$ can be written as
        \begin{equation}
            \mc{J}_\omega = -\pi^*\omega + \hat{\pi}^*\omega^{-1}
        \end{equation}
        which is the K\"ahler form associated to the product K\"ahler
        manifold $(X,g,J,-\omega)\times (\hat{X},g^{-1},-J^*,\omega^{-1})$.
    \end{example}
    
    Even if there is no metric on $X$, the doubling space comes with a canonical
    metric. With respect to this metric, a construction similar to
    \cref{ex:liftsKahler} can be performed.

    Let $\mc{J}$ be the lift to $\mathbb{X}$ of a generalized complex
    structure on $X$. Then, the form
    \begin{equation}
        \mc{J}^\sharp := \la \mc{J}\blank, \blank\ra:T\mathbb{X}\to T^*\mathbb{X}
    \end{equation}
    is a nondegenerate two-form on $\mathbb{X}$.
    \begin{example}
        \label{ex:liftOfComplexStructureSymplectic}
        Let $J$ be the complex structure on $X$, and $\mc{J}$ its lift. 
        Since $J$ is invariant under
        translation on $X$, the form $\la J\blank,\blank\ra$ is likewise
        an invariant form on $\mathbb{X}$. Thus this form is closed
        and $\mc{J}^\sharp$ defines a symplectic form on $\mathbb{X}$.
    \end{example}

\subsection{Lifting D-branes to the Doubling Space}
    
In this context, a (rank one) \textit{D-brane} is a pair $(M,\nabla)$ of smooth
submanifold $M$ along with a $U(1)$-bundle with $U(1)$ connection $\nabla$. If the 
submanifold $M$ is holomorphic and the connection is compatible with
the holomorphic structure, $(M,\nabla)$ defines a (rank one) \textit{B-type
D-brane}, whereas if $M$ is Lagrangian and $\nabla$ is flat, then $(M,\nabla)$
defines a (rank one) \textit{A-type D-brane}. B-type (A-type) D-branes are
natural boundary conditions for the topological B-model (A-model).

In \cite{qin2020coisotropic}, a lift of certain rank one A-type D-branes is
defined on a symplectic torus, which we review now.
Let $X$ be a symplectic torus, and $ L $ a Lagrangian submanifold of $ X $,
endowed with a flat $U(1)$ connection $\nabla$, or
a rank-one coisotropic brane in the sense of \cite{Kapustin:2001ij}.
Furthermore, fix a (set-theoretic) map
\begin{equation}
    \xi:H_1(L)\to \Z_2
\end{equation}
satisfying 
\begin{equation}
    \xi(\gamma+\gamma')-\xi(\gamma)-\xi(\gamma') = c_1(\nabla)(\gamma\wedge\gamma')
\end{equation}
in $\Z_2$.
Then,
\begin{definition}[\protect{\cite{qin2020coisotropic}*{Definition 3.4}}]
    The lift of $L$ is defined to be
    \begin{equation}
        \{(x,\hat{x})\in X\times\hat{X}\ |\ x\in L,
            \la \hat{x},\gamma_x\ra = (-1)^{\xi(\gamma_x)}\hol_\nabla(\gamma_x),
        \forall\gamma_x\in\pi_1(X,x)\}
    \end{equation}
    where $\gamma_x$ is any linear circle based at $x$.
\end{definition}
Y.\ Qin then computes the intersection theory of these lifts, and shows that
their Floer intersection cohomology is closely related to the Fukaya
category of $X$ itself. We now attempt to carry this construction over to the
B-model side where $X$ is a complex torus.

To that end, we now let $(M,\nabla)$ be a D-brane whose submanifold $M$ is a
(closed) holomorphic subtorus,
with $F_\nabla = \nabla^2$ denoting the curvature two-form which we require to be
translation-invariant. The embedding of $\Lambda$ into $X$ defines linear
representatives $\gamma$ for each element in $H_1(M,\Z)$. 
\begin{definition}
    \label{def:xi}
    A \textit{symmetric semi-character} for $E\in IF^2(M)$ is a (set-theoretic) map
    \begin{equation}\label{xiDefinition}
        \xi:\Lambda \to \Z_2
    \end{equation}
    satisfying
    \begin{equation}\label{xiConstraint}
        \xi( \gamma_1 + \gamma_2) -\xi(\gamma_1)-\xi(\gamma_2) =
        E(\gamma_1\wedge \gamma_2)
    \end{equation}
    for all $\gamma_{1,2}\in H_1(M,\Z)$. Here we implicitly use the isomorphism
    \eqref{eq:cohomologyIsIF}. 
\end{definition} 

\begin{proposition}
    \label{prop:existenceOfXi}
    For each invariant two-form $E$, symmetric semi-characters exist,
    and the space of symmetric semi-characters is an affine space
    over $H^1(M,\Z_2)$.
\end{proposition}
\begin{proof}
    Let $\{s_i\}$ freely generate $H_1(M,\Z)$, and arbitrarily fix values for
    $\xi(s_i)$. Then, formula  \eqref{xiConstraint} determines the value of $\xi$
    on all linear combinations of $s_i$, hence all of $H_1(M,\Z)$. 
    In particular, since $E(s_i\wedge s_i)=0$, $\xi(ns_i) = n\xi(s_i)$,
    the value of $\xi$ is computed inductively as
        \begin{equation}
            \xi(a^is_i) = a^i\xi(s_i) + a^ia^jE(s_i\wedge s_j)
        \end{equation} 
    which satisfies the desired identity.

    The difference of any two such functions, say $\xi_1,\xi_2$, satisfies
        \begin{equation}
            (\xi_1-\xi_2)(\gamma_1+\gamma_2) = \xi_1(\gamma_1) + \xi_1(\gamma_2)
            -\xi_2(\gamma_1)-\xi_2(\gamma_2) = (\xi_1-\xi_2)(\gamma_1) +(\xi_1-\xi_2)(\gamma_2)
        \end{equation} 
    hence is an element of $H^1(M,\Z_2)$. 
    Conversely, adding an element $\sigma\in H^1(M,\Z_2)$ to $\xi$ clearly results
    in a new function satisfying the desired identity.
\end{proof}

\begin{proposition}
    For any choice of $\xi$, the map
        \begin{equation}
            v \mapsto (-1)^{\xi(v)}\hol_\nabla(x+\gamma_v)
        \end{equation} 
    is a group homomorphism from $H_1(M,\Z)$ to $U(1)$.
\end{proposition}
\begin{proof}
    The obstruction to $\hol_\nabla$ being a group homomorphism is measured by
    the quotient
        \begin{equation}
            q(v_1,v_2) = \frac{\hol_\nabla(x+\gamma_{v_1+v_2})}
            { \hol_\nabla(x+\gamma_{v_1}) \hol_\nabla(x+\gamma_{v_2})}
        \end{equation} 
    which in general is not $1$ if the curvature of $\nabla$ is nonzero. Using the
    local expression for the holonomy along a path $\gamma$ as
        \begin{equation}
            \hol_\nabla\left(\gamma\right) = \exp\left(2\pi i \int_\gamma A\right)
        \end{equation} 
    for $A$ a local gauge potential, we see that Stokes' theorem guarantees
        \begin{equation}
            q(v_1,v_2) = \exp\left(2\pi i\int_{\Sigma_{1,2}} F_\nabla\right)
        \end{equation} 
    where $\Sigma_{1,2}$ is the triangular region bounded by $\gamma_{v_1+v_2},\gamma_{-v_2}$
    and $\gamma_{-v_1}$ in that order. 

    The map $v\mapsto (-1)^{\xi(v)}\hol_\nabla(x+\gamma_v)$ is a group homomorphism
    when the quotient
        \begin{equation}
            Q(v_1,v_2) = (-1)^{\xi(v_1+v_2)-\xi(v_1)-\xi(v_2)}q(v_1,v_2)
        \end{equation} 
    is unity. However, since $\xi$ is a semi-character, this quotient evaluates to
        \begin{equation}
            \begin{aligned}
                Q(v_1,v_2) &= (-1)^{F_\nabla(v_1,v_2)}\exp\left(2\pi i
                \int_{\Sigma_{1,2}}F_\nabla\right)\\
                           &= \exp\left(\pi i \left( F_\nabla(v_1,v_2) +
                           2\int_{\Sigma_{1,2}}F_\nabla \right)\right)
            \end{aligned}
        \end{equation}.
    Since $F_\nabla$ is an invariant form,
    \begin{equation}
        \int_{\Sigma_{1,2}}F_\nabla = \frac{1}{2}F_\nabla(v_1,v_2)
    \end{equation}
    resulting in
    \begin{equation}
        Q(v_1,v_2) = 1
    \end{equation}
    as desired.
\end{proof}

An element $\hat{x}\in \hat{X}$ of the dual torus naturally provides a group
homomorphism $H_1(X,\Z)\to U(1)$ in the following way. Choose any lift
$\hat{v}\in H^1(X,\R)$ of $\hat{x}$ to the universal cover, and for any $v\in
H_1(X,\Z)$ define
\begin{equation}
    \hat{x}(v) = \exp(2\pi i \hat{v}(v))
\end{equation} 
as an element of $U(1)$. Notice that any two choices of $\hat{v}$ differ by an
element of $H^1(X,\Z)$ the dual lattice of $H_1(X,\Z)$. Hence, shifting
$\hat{v}$ by an element of $H^1(X,\Z)$ shifts the value of $v$ by an integer,
leaving the expression for $\hat{x}$ unchanged.

\begin{definition}
    The \textit{doubled lift} of $(M,\nabla)$ is defined to be the submanifold
        \begin{equation}
            \mathbb{M} = \{ (x,\hat{x})\in X\times \hat{X}\ |\ x\in M,
                \hat{x}(v) = (-1)^{\xi(v)}\hol_\nabla(x + \gamma_v) \forall
                v\in H_1(L,\Z)\}
        \end{equation} 
\end{definition} 
\begin{remark}
    This is only well-defined if $\nabla$ is $F_\nabla$-flat, but this is
    not as restrictive as one might think. Indeed, every holomorphic line bundle
    admits a unique $c_1(\mc{L})$-flat connection, and computations
    in the topological B-model show that the holomorphic structure
    of the line bundle is the only relevant data to compute the open string
    spectrum.
\end{remark}

\subsubsection{Lifting Cheeger-Simons Differential Characters}

The lift has a different interpretation in terms of Cheeger-Simons differential
characters, which characterize principal $U(1)$-bundles with connection. That
is,
\begin{proposition}
    There is a one-to-one correspondence between space-filling rank-one D-branes
    on $X$ and the second differential cohomology group $\hat{H}^2(X)$ of $X$.
    Furthermore, the tensor product of D-branes corresponds to the addition
    law on differential cohomology.
\end{proposition} 
\begin{proof}
    This follows immediately from \cref{DCtoLineBundleIsAdditive}.
\end{proof}

\begin{definition}
    An $E$-flat line bundle is called a \textit{symmetric $E$-flat line bundle}
    if the holonomy around elements of $\Lambda$ (thought of as linear cycles through
    the origin) is contained in $\{\pm 1\}$. Equivalently,
    an $E$-flat line bundle is a symmetric $E$-flat line bundle if the holonomy
    is given by the exponential of a symmetric semi-character (\cref{def:xi}).
\end{definition} 

\begin{proposition}\label{xiConnectionsExistence}
    Fix an invariant $2$-form $E\in IF^2(X)$.
    For each $\xi$ a symmetric semi-character as in \cref{def:xi}, there exists
    a unique symmetric $E$-flat line bundle with holonomy around elements
    of $\Lambda$ given by $(-1)^\xi$.
\end{proposition} 

\begin{proof}
    The association
    \begin{equation}
        \lambda\mapsto (-1)^{\xi(\lambda)}
    \end{equation}
    defines a map from the linear $1$-cycles $\Lambda$ on $X$ to $U(1)$.
    This map satisfies the condition \eqref{curvatureExistence} on linear cycles
    exactly due to the constraint \eqref{xiConstraint} on $\xi$.

    From the proof of \cref{thm:structureDConTori}, we see that such a $\xi$
    defines a unique differential character $f\in\hat{H}^2(X)$ satisfying
    \begin{equation}
        f(\lambda) = (-1)^{\xi(\lambda)}
    \end{equation}
    with
    \begin{equation}
        \curv(f) = E\in IF^2(X)
    \end{equation}.
    Using the isomorphism between $E$-flat connections and $\Pic^E(X)$, we associate
    to $f$ a line bundle with connection $(S_E,\nabla)$ with the prescribed
    holonomy.
\end{proof}

\begin{theorem}
    \label{thm:liftDescriptionDC}
    Let $\mc{L}$ be a line bundle, $E=c_1(\mc{L})$, and $\nabla$ an $E$-flat
    connection on $\mc{L}$. Let $S_{-E}$ be a symmetric $-E$-flat line bundle
    (notice the minus sign!) defined in \cref{xiConnectionsExistence}.
    Then, the lift of $(\mc{L},\nabla)$ is
    equivalently described as the graph
    \begin{equation}\label{liftOfDifferentialCharacter}
        \mathbb{L}(\mc{L}) = \Gamma\left(x\mapsto (S_{-E}\otimes t_{-x}^*\mc{L})\in
        \hat{X}\right)
    \end{equation}
    where $\hat{X}\cong\Pic^0(X)$ in the natural way.
\end{theorem}

\begin{example}[Lifting the structure sheaf]
    \label{ex:liftStructure}
    Consider the trivial D-brane given by $(X,d)$ with $d$ the exterior derivative,
    acting on the trivial line bundle $X\times \C$. The sheaf of holomorphic sections
    of this line bundle is $\oh_X$. This lift plays special significance, so
    we denote it by $\mathbb{O}_X$.
    
    Since the connection is flat, i.e. $d^{2}=0$, and has trivial monodromy by
    the definition of the lift we find the lift is
        \begin{equation}
            \begin{aligned}
                \mathbb{O}_X := \mathbb{L}({\oh_X}) &= \{(x,\hat{x})\ |\
                (-1)^{\xi(\gamma_i)}\hat{x}(\gamma_i)=1\}\\ &= \{(x, \hat{x})\
            |\ e^{\pi i (\xi(\gamma_i) + 2\hat{v}(\gamma_i)} = 1\}\\
            \end{aligned}
        \end{equation} 
    where $\hat{v}\in H^1(X,\R)$. From this, we see that 
        \begin{equation}
            \frac{1}{2}\xi(\gamma_i) = \hat{v}(\gamma_i) \text{ (mod $\Z$)}
        \end{equation} 
    so that $\hat{v}$ is determined by $\xi$ in $H^1(X,\R)/H^1(X,\Z)$. This
    therefore determines a point $\hat{x}\in \hat{X}$ and the lift of $\oh_X$ is
    \begin{equation}
        \mathbb{O}_X = X\times \{\left[ e^{\pi i \xi} \right]\}
    \end{equation}.

    Using the language of differential characters, this line bundle corresponds
    to the trivial character $f=1$.  
    We still need to choose a $\Z_2$-character $\xi$, which corresponds
    to a choice of symmetric flat line bundle $S_0$ on $X$. Its square is
    trivial, hence this corresponds to a choice of $2$-torsion point in
    $\hat{X}$. The lift is then
    \begin{equation}
        \begin{aligned}
            \mathbb{O}_X &= \Gamma(x\mapsto S_0\otimes t_{-x}^*\oh_X)\\
                                &=\Gamma(x\mapsto S_0)\\
                                &= X\times \{\left[ S_0 \right]\}
        \end{aligned} 
    \end{equation}
    as computed above. That is, the lift is the fiber over the point $S_0$.
\end{example}

\begin{example}[Lifting other flat line bundles]
    \label{ex:liftFlat}
    Consider now the D-brane given by $(X,\nabla)$ where $\nabla$ is another
    flat connection on the trivial line bundle $X\times\C$. Since $\nabla$ is
    flat, the holonomy of $\nabla$ is given by monodromy, represented by an element
    $\alpha\in H^1(X,U(1))\cong \hat{X}$. Denote by $\mc{L}$ the corresponding
    line bundle.

    Choose a $\xi$ as in \cref{ex:liftStructure}, and denote the corresponding
    symmetric line bundle as $S_0$. Notice that since $\mc{L}$ is flat,
    it is translation-invariant.
    The lift of $\mc{L}$ is then
    \begin{equation}
        \begin{aligned}
            \mathbb{L}(\mc{L}) &= \Gamma(x\mapsto S_0\otimes t_{-x}^*\mc{L})\\
                               &=\Gamma(x\mapsto S_0\otimes\mc{L})\\
                               &= X\times \{\left[ S_0\otimes\mc{L} \right]\}
        \end{aligned} 
    \end{equation}
    which, in terms of characters, is the fiber over the point $(-1)^\xi \alpha$.
\end{example}

\begin{example}[Lifts of $E$-flat Connections]
    \label{ex:liftEFlat}
    Let $(X,\nabla)$ be a space-filling D-brane with $F_\nabla = i\pi E$ for some
    fixed $E\in IF^2(X)$ closed with integral periods.

    Using the isomorphism in \cref{thm:EFlatIsHolomorphicIsomorphism}, we associate
    to $\nabla$ a holomorphic line bundle $\mc{L}$ with $c_1(\mc{L}) = E$. 
    We choose as well an $-E$-flat line bundle $S_{-E}$ as in
    \cref{xiConnectionsExistence}.
    Then, the lift is defined to be
    \begin{equation}
        \mathbb{L}(\mc{L}) = \Gamma(x\mapsto S_{-E}\otimes t_{-x}^*\mc{L})
    \end{equation}
    as before.

    More can be said about this map. The line bundle $S_{-E}\otimes t_{-x}^*\mc{L}$
    can be rewritten as
    \begin{equation}
        t_x^*S_{-E}\otimes S_{-E}^{-1}\otimes(\mc{L}\otimes S_{-E})
    \end{equation}
    which, using the morphism \eqref{eq:phiLDefn}, is just
    \begin{equation}
        \phi_{S_{-E}}(x)+\left[ \mc{L}\otimes S_{-E} \right]
    \end{equation},
    a translate of $\phi_{S_{-E}}$.
\end{example}

\subsection{Geometry of Doubled Branes}
    To simplify discussion, choose once and for all a symmetric $E$-flat line
    bundle $S_E$ for each $E$, in particular choosing $\oh_X$ for $E=0$. 
    We take these symmetric $E$-flat line bundles as part of the data for the lift.
    The choice is arbitrary up to a translation by a $2$-torsion point in $\hat{X}$.

    \subsubsection{Properties of the Lift}
    
    Similar to what was found in \cite{qin2020coisotropic}, the lifted submanifold
    enjoys many nice geometric properties. As before, let $\mc{L}$ be a line
    bundle on $X$ and denote its first Chern class by $E := c_1(\mc{L})$.
    Let $\mathbb{L}(\mc{L})$ denote the lift of $\mc{L}$ described in
    \cref{thm:liftDescriptionDC}.

    \begin{proposition}[\protect{c.f.\ \cite{qin2020coisotropic}*{Proposition~3.6} in the
        case of A-branes}]
        \label{prop:tangentBundleLift}
        The tangent bundle of $\mathbb{L}(\mc{L})$ is the subbundle of $T\mathbb{X}$
        defined by
        \begin{equation}
            \label{eq:tangentBundleLift}
            T\mathbb{L}(\mc{L}) = \{ (v,\hat{v})\in TX\oplus T\hat{X}\ |\ 
            \iota_{\hat{v}}\sigma = \iota_v E\}
        \end{equation}.
        That is, under the isomorphism $T\mathbb{X}\cong \pi^*(TX\oplus T^*X)$ the
        tangent bundle of $\mathbb{L}(\mc{L})$ is isomorphic to the pullback of
        the generalized tangent bundle of $(X,E)$.
    \end{proposition} 
    \begin{proof}
        Using the description in \cref{thm:liftDescriptionDC} of the lift
        as the graph of the function
        \begin{equation}
            \phi := \phi_{S_{-E}} + \left[ \mc{L}\otimes S_{-E} \right]
        \end{equation}
        the tangent space can be computed as the graph of the differential
        $d\phi$. However, $d\phi = d\phi_{S_{-E}}$ and $\phi_{S_{-E}}$
        has a linear analytic representation
        \begin{equation}
            \begin{aligned}
                \phi_{-E}:V&\to\overline{\omega}\\
                v\mapsto H(v,\blank)
            \end{aligned} 
        \end{equation}
        whose derivative at $(x,\phi(x))$, since $\phi_{-E}$ is linear, is
        evidently
        \begin{equation}\label{eq:differentialOfLiftMap}
            \begin{aligned}
                d\phi_{-E}:T_xX&\to T_{\phi(x)}\hat{X}\cong T^*_xX\\
                v&\mapsto (w\mapsto H(v,w))
            \end{aligned} 
        \end{equation}.
        Comparing this to the computation in the proof of \cref{thm:EFlatIsHolomorphicIsomorphism}
        we recover the description of the tangent space as the graph
        of
        \begin{equation}
            v\mapsto \iota_v E 
        \end{equation}
        as desired.
    \end{proof} 

    The construction of \cref{prop:tangentBundleLift} reveals that
    many of the properties that the generalized tangent bundle of $X$ enjoys
    are lifted to the doubling space. We illustrate some of these now.
    \begin{proposition}
        \label{prop:liftIsMaxIsotropic}
        The lift $\mathbb{L}(\mc{L})$ of a line bundle $\mc{L}$ is a maximal
        isotropic submanifold of $\mathbb{X}$ under the canonical neutral
        metric.
    \end{proposition} 
    \begin{proof}
        At a point $(x,\phi(x))$ of $\mathbb{L}(\mc{L})$, the tangent
        space, by \eqref{eq:tangentBundleLift}, consists of vectors
        of the form
        \begin{equation}
            T\mathbb{L}(\mc{L}) = \{ v + \hat{v}\ |\ \iota_{\hat{v}}\sigma = \iota_v E\}
        \end{equation}
        and thus if $v+\hat{v}$ and $w+\hat{w}$ are two tangent vectors to $\mathbb{L}(\mc{L})$,
        \begin{equation}
            \begin{aligned}
                \sigma(v+\hat{v},w+\hat{w})
                &= \sigma(\hat{v},w) + \sigma(\hat{w},v)\\
                &= \iota_{\hat{v}}\sigma(w) + \iota_{\hat{w}}\sigma(v)\\
                &= E(v,w) + E(w,v)
            \end{aligned} 
        \end{equation}
    \end{proof} 

    \begin{proposition}
        Let $\mc{J}$ be the complex structure on $\mathbb{X}$ induced by the
        complex structure on $X$, and let $\mc{J}^\sharp$ be
        the induced symplectic form as in
        \cref{ex:liftOfComplexStructureSymplectic}. Then, the lift
        $\mathbb{L}(\mc{L})$ is both $\mc{J}$-holomorphic
        and $\mc{J}^\sharp$-Lagrangian.
    \end{proposition}
    \begin{proof}
        Since the tangent bundle of $\mathbb{L}(\mc{L})$ is identified
        with the generalized tangent bundle of $(X,E)$, the endomorphism
        $\mc{J}$ acts on $\mathbb{L}(\mc{L})$ as in
        \cite{gualtieri2004generalized}*{Example~7.7}. This is $\mc{J}$-holomorphic
        since $E$ is of type $(1,1)$ on $X$.

        This property, combined with \cref{prop:liftIsMaxIsotropic} implies
        that for every $v,w\in T\mathbb{L}(\mc{L})$,
        \begin{equation}
            \mc{J}^\sharp(v,w) = \la \mc{J}v, w\ra = 0
        \end{equation}
        since $\mc{J}v\in T\mathbb{L}(\mc{L})$ and $T\mathbb{L}(\mc{L})$ is
        isotropic under the metric. 
        Thus, $\mc{J}^\sharp$ vanishes on $\mathbb{L}(\mc{L})$ and $\mathbb{L}(\mc{L})$,
        being middle-dimensional, is Lagrangian.
    \end{proof}

    The intersections of the lifts of all line bundles on $X$
    can also be computed using the results of \cref{SSS:cohomologyOfLineBundles}.
    Let $M_L(\mathbb{X})$ denote the set of Lagrangian subspaces
    of $\mathbb{X}$, which carries an action of $X$ and $\hat{X}=\Pic^0(X)$
    given by translation.
    \begin{proposition}
        \label{prop:liftsAreEquivariant}
        The map $\Pic(X)\to {M}_L(X\times\hat{X})$ sending a line
        bundle to its lift is $\Pic^0(X)$-equivariant.
    \end{proposition} 
    \begin{proof}
        This follows immediately from the definition of the lift. Let $\mc{L}$ be
        arbitrary and $\mc{L}_0\in\Pic^0(X)$, represented by a point $\hat{x}\in\hat{X}$.
        Then, the lift of their tensor product is
        \begin{equation}
            \begin{aligned}
                \mathbb{L}(\mc{L}_0\otimes\mc{L})
                &= \Gamma(x\mapsto t_x^*S_{-E}\otimes\mc{L}\otimes\mc{L}_0)\\
                &=\Gamma(x\mapsto \phi_{S_{-E}}(x) + \left[ \mc{L}\otimes
                S_{-E}\otimes\mc{L}_0 \right])\\
                &=\mathbb{L}(\mc{L}) + \left[ \mc{L}_0 \right]
            \end{aligned} 
        \end{equation}
        as desired.
    \end{proof} 

    \begin{proposition}
        \label{prop:intersectionLiftUnderTensor}
        Let $\mc{L}_1,\mc{L}_2$ and $\mc{L}$ be arbitrary line bundles on $X$.
        Then, the operation $\blank\otimes\mc{L}$ leaves the intersection invariant
        up to a shift by $\Pic^0(X)$.
        That is,
        \begin{equation}
            \mathbb{L}(\mc{L}_1)\cap\mathbb{L}(\mc{L}_2)
            =
            t_{\hat{x}}(\mathbb{L}(\mc{L}_1\otimes\mc{L})\cap
            \mathbb{L}(\mc{L}_2\otimes\mc{L}))
        \end{equation}
        for some $\hat{x}\in\hat{X}$,
        as submanifolds of $\mathbb{X}$.
    \end{proposition}
    \begin{proof}
        Let $E_i=c_1(\mc{L}_i)$ and $E=c_1(\mc{L})$, recalling that
        $S_{-E_i}$ is the symmetric $-E_i$-flat line bundle fixed at the beginning
        of this section.
        If $(x,\hat{x})$ lies in the intersection of
        $\mathbb{L}_1:=\mathbb{L}(\mc{L}_1)$ and
        $\mathbb{L}_2:=\mathbb{L}(\mc{L}_2)$,
        then by definition
        \begin{equation}
            \hat{x} = S_{-E_1}\otimes t_{-x}^*\mc{L}_1
            =S_{-E_2}\otimes t_{-x}^*\mc{L}_2
        \end{equation}.
        After tensoring with $\mc{L}\otimes S_{-E}$, equality still holds at $x$:
        \begin{equation}
            S_{-E_1}\otimes t_{-x}^*\mc{L}_1\otimes\mc{L}\otimes S_{-E}
            =
            S_{-E_2}\otimes t_{-x}^*\mc{L}_2\otimes\mc{L}\otimes S_{-E}
        \end{equation}
        but this may no longer be the line bundle $\hat{x}$, as it has been
        shifted by the flat line bundle $\mc{L}\otimes S_{-E}$.
    \end{proof}

    Using these properties, as well as the results of
    \cref{SSS:cohomologyOfLineBundles},
    the intersections of lifts of arbitrary line bundles can be computed.

    \begin{theorem}
        \label{thm:liftsDoNotIntersect}
        Let $\mc{L}_1$ and $\mc{L}_2$ be two holomorphic line bundles with
        $c_1(\mc{L}_1) = c_1(\mc{L}_2)$. Then, their lifts do not intersect unless
        $\mc{L}_1=\mc{L}_2$. 
    \end{theorem} 
    \begin{proof}
        This follows immediately from \cref{prop:liftsAreEquivariant}, since
        $\mc{L}_2 = \mc{L}_0\otimes\mc{L}_1$ for some $\mc{L}_0\in\Pic^0(X)$.
    \end{proof} 

    \begin{theorem}
        \label{thm:computingIntersectionArbitraryLifts}
        Let $\mc{L}_1$ and $\mc{L}_2$ be arbitrary line bundles,
        with lifts $\mathbb{L}_1$ and $\mathbb{L}_2$. Then,
        their intersection is either empty or is given by
        \begin{equation}
            \mathbb{L}_1\cap\mathbb{L}_2 = K(\mc{L}_2\otimes\mc{L}_1^{-1})
        \end{equation}
        up to translations in $\hat{X}$ and $X$,
        where $K(\mc{L}_2\otimes\mc{L}_1^{-1})$ is the kernel of the map
        defined in \eqref{eq:phiLDefn}, thought of as a submanifold of $\mathbb{X}$
        via the zero section $\mathbb{O}_X$.
    \end{theorem}
    \begin{proof}
        Let $E_i=c_1(\mc{L}_i)$, and recall that $S_{E_i}$ is the chosen
        symmetric $E_i$-flat line bundle. 
        By \cref{prop:intersectionLiftUnderTensor}, up to a shift in $\hat{X}$
        the intersection can equivalently be computed as
        \begin{equation}\label{eq:intersectionStructureLiftWithLineBundle}
            \begin{aligned}
                \mathbb{L}_1\cap\mathbb{L}_2
                &=\mathbb{L}(\mc{L}_1\otimes\mc{L}_1^{-1})
                \cap \mathbb{L}(\mc{L}_2\otimes\mc{L}_1^{-1})\\
                &=\mathbb{O}_X\cap\mathbb{L}(\mc{L}_2\otimes\mc{L}_1^{-1})
            \end{aligned}
        \end{equation}.
        
        As was noted in \cref{ex:liftEFlat}, the lift of
        $\mc{L}_2\otimes\mc{L}_1^{-1}$ is the graph of the map
        \begin{equation}
            \phi:= \phi_{S_{-E_2+E_1}} + \hat{x}_1
        \end{equation}
        for
        \begin{equation}
            \hat{x}_1 = \left( \mc{L}_2\otimes\mc{L}_1^{-1} \right)
            \otimes S_{-E_2+E_1}
        \end{equation}
        . 
        Since $\mathbb{O}_X$ is the zero section, the intersection to be computed
        \eqref{eq:intersectionStructureLiftWithLineBundle}
        is the kernel of this map, which is explicitly computed as the
        set of $x\in X$ satisfying
        \begin{equation}
            \phi(x) = -\hat{x}_1
        \end{equation}.
        But since $\phi$ is a group homomorphism, the set $\phi^{-1}(\hat{x}_1)$
        of solutions is either empty or a translate in $X$ of the kernel of $\phi$.
    \end{proof}

    \begin{remark}
        The case of \ref{thm:liftsDoNotIntersect} is a special case of        
        \ref{thm:computingIntersectionArbitraryLifts} when $E_2 = -E_1$.
        In this case,
        \begin{equation}
            \phi = \phi_{\oh_X} + \left( \mc{L}_2\otimes\mc{L}_1^{-1} \right)
        \end{equation}
        and since $\phi_{\oh_X}$ is the zero map, we recover the desired result.
    \end{remark}

    \subsubsection{Relation to Physics: $\Ext$-groups}
    \label{SSS:relationPhysicsExt}
    
    It is expected that the intersection theory of the doubled lifts of
    rank-one B-type D-branes is closely related to the corresponding open string
    spectrum. Specifically, in \cite{qin2020coisotropic} it was noted that
    the complex structure on $\mathbb{X}$ induced by the symplectic structure
    on $X$ selects a subspace of the intersection Floer cohomology on the lifts
    which computes the A-model open string spectra. 

    We now show that, in certain cases, a similar phenomenon happens to the
    intersection theory of lifts of B-type D-branes with respect to the induced
    complex structure. Namely, the total $Ext$-group between the two line bundles
    can be computed in this way. Recall (\cref{thm:katzSharpeExtGroups}) the
    B-model $\Hom$-spaces are computed using the total $\Ext$-group. We
    denote by $\HF^*$ the Lagrangian intersection Floer cohomology.

    \begin{definition}[\protect{c.f.\ \cite{qin2020coisotropic}*{Definition~5.1}}]
        Let $\mc{J}$ be the lift of the complex structure, defined in \cref{ex:liftComplexStructure}.
        The \textit{$\mc{J}$-holomorphic part} of the Floer cohomology ring
        $\HF^*(\mathbb{L},\mathbb{L})\cong H^*(\mathbb{L},\C)$ for any lift
        $\mathbb{L}$ is the $(0,*)$-part of $H^*(\mathbb{L},\C)$ under
        the Hodge decomposition with respect to the complex structure $\mc{J}$.
    \end{definition} 

    \begin{theorem}
        \label{thm:ExtIsIntersectionEqualChern}
        Let $\mc{L}_1$ and $\mc{L}_2$ be two holomorphic line bundles with
        $c_1(\mc{L}_1) = c_1(\mc{L}_2)$. Then,
        \begin{equation}
            \Hom_B(\mc{L}_1,\mc{L}_2) 
            = \HF^*_\mc{J}(\mathbb{L}(\mc{L}_1),\mathbb{L}(\mc{L}_2)
        \end{equation}
        where $\HF^*_\mc{J}$ is the $\mc{J}$-holomorphic part of $\HF^*$,
        and $\mc{J} = \mc{J}_J$ is the complex structure on $\mathbb{X}$
        induced by the complex structure on $X$ (\cref{ex:liftComplexStructure}).
    \end{theorem} 
    \begin{proof}
        The B-model morphism space is computed via Ext-groups, i.e.
        \begin{equation}
            \Hom_B(\mc{L}_1,\mc{L}_2) = \bigoplus_q\Ext^q(\mc{L}_1,\mc{L}_2)
        \end{equation}
        which vanishes unless $\mc{L}_1=\mc{L}_2$. By \cref{thm:liftsDoNotIntersect},
        the lifts of these two line bundles do not intersect unless they
        are equal as well.
        When they are equal, by tensoring with their inverse we find that
        \begin{equation}
            \Hom_B(\mc{L}_1,\mc{L}_2) = \Ext^*(\oh_X,\oh_X) = H^*(\oh_X)
        \end{equation}
        which we have computed in \cref{rk:cohomologyofStructure}.
        
        Let $\mathbb{L} = \mathbb{L}(\mc{L}_1)$ be the lift. The Floer
        cohomology of $\mathbb{L}$ is well-known to be
        \begin{equation}
            \HF^*(\mathbb{L},\mathbb{L}) \cong H^*(\mathbb{L},\C)\cong H^*(X,\C)
        \end{equation}.
        As expected, the $J$-holomorphic part
        \begin{equation}
            \HF^*_J(\mathbb{L},\mathbb{L})\cong H^{0,*}(\mathbb{L},\C)\cong H^*(\oh_X)
        \end{equation}
        computes the desired ext-group.
    \end{proof}

\section{Conclusions and Outlook} %---

The ideas of T-duality suggest that the geometry arising from string theory on
a torus can be studied using a corresponding theory on the product of a torus
and its dual. We have seen this bear out in the boundary theory of the topological
B-model by defining the doubled lift of a space-filling rank-one $E$-flat
D-brane and studying the properties of this lift.

The geometry of the doubling space was studied, and it was shown to have
a canonical neutral metric and fundamental two-form. Furthermore, we saw
that the complex structure on the original torus lifted to a complex
structure on the doubling space which, paired with the neutral metric, induced
a symplectic structure on the doubling space. The lifts of D-branes were shown
to be both holomorphic under the induced complex structure and
maximally isotropic under the canonical metric, making them Lagrangian with
respect to the induced symplectic form.

The intersection theory of these Lagrangian lifts was shown, in the case of
two line bundles of the same topological type, to compute the expected
$\Hom$-space of the topological B-model. This agrees with the parallel
analysis of \cite{qin2020coisotropic} in the A-model case. 
We also showed that the intersection properties of these Lagrangian lifts
generally is described by the kernel of the well-known morphism of
\eqref{eq:phiLDefn}.

Because the lift is $\Pic^0(X)$-equivariant up to translations, computations
of arbitrary intersections reduce to intersections of a modified lift
with the zero section $\mathbb{O}_X$ defined in \cref{ex:liftStructure}.
The topological B-model $\Hom$-space is expected to be given by
\begin{equation}
    \Hom_B(\oh_X,\mc{L}) = \bigoplus_q \Ext^q(\oh_X,\mc{L})\cong H^q(X,\mc{L})
\end{equation}.
In the case of $\mc{L}$ being positive-definite, the intersection of its lift
$\mathbb{L}(\mc{L})$ with $\mathbb{O}_X$ is zero-dimensional
and cohomology of $\mc{L}$ is concentrated in degree zero. The na\"ive
expectation that the cardinality of the intersection agree with $h^0(\mc{L})$,
however, is incorrect. In fact, comparing
\cref{thm:cohomologyOfLineBundles} with \eqref{eq:cardinalityOfKernelPhiL}
we see that
\begin{equation}
    \#(\mathbb{L}(\mc{L})\cap\mathbb{O}_X) = \left( h^0(\mc{L}) \right)^2
\end{equation}.
One possible hint to resolving this lies in \cite{kapustin2005open} where it was
shown that the open string spectrum for a generalized complex brane is
computable using Courant algebroid cohomology. One might expect that
this cohomology has a geometric analogue in the doubling space. We leave this
question, however, for future work.

%----- Bibliography ----------------


\end{document}